\documentclass[]{elsarticle}

\usepackage{lipsum}
\usepackage{xcolor}
\makeatletter
\def\ps@pprintTitle{%
 \let\@oddhead\@empty
 \let\@evenhead\@empty
 \def\@oddfoot{}%
 \let\@evenfoot\@oddfoot}
\makeatother

\usepackage{lscape}
\usepackage{lineno,hyperref}
\usepackage{rotating}
\modulolinenumbers[5]
\usepackage{amssymb,amsmath, amsthm,epsfig,multirow,graphicx}
\usepackage{Tabbing}
\usepackage{bibentry}
\usepackage[latin1]{inputenc}
\newcommand\scalemath[2]{\scalebox{#1}{\mbox{\ensuremath{\displaystyle #2}}}}

\newcommand{\ignore}[1]{}

\newcommand{\eqm}{\begin{eqnarray}}
\newcommand{\enm}{\end{eqnarray}}
\newcommand{\eql}[1]{\begin{equation}\label{#1}}

\newcommand{\eqml}[1]{\eql{#1}\begin{array}{rcl}}
\newcommand{\enml}{\end{array}\end{equation}}

\newcommand{\eqmno}{\begin{eqnarray*}}
\newcommand{\enmno}{\end{eqnarray*}}

\def\dsp{\displaystyle}

\newtheorem{theorem}{Theorem}

\newtheorem{remark}{Remark}

\newtheorem{lemma}{Lemma}
\newtheorem{corollary}{Corollary}

\def \bi \begin{itemize}
\def \ei \end{itemize}
\def \dsp \displaystyle

\journal{}

\makeatletter
\def\@author#1{\g@addto@macro\elsauthors{\normalsize%
    \def\baselinestretch{1}%
    \upshape\authorsep#1\unskip\textsuperscript{%
      \ifx\@fnmark\@empty\else\unskip\sep\@fnmark\let\sep=,\fi
      \ifx\@corref\@empty\else\unskip\sep\@corref\let\sep=,\fi
      }%
    \def\authorsep{\unskip,\space}%
    \global\let\@fnmark\@empty
    \global\let\@corref\@empty  
    \global\let\sep\@empty}%
    \@eadauthor={#1}
}
\makeatother

\usepackage{mathptmx} 
\usepackage{siunitx} 
\usepackage{tikz} 
\usetikzlibrary{decorations.pathmorphing} 
\usetikzlibrary{arrows.meta} 
\colorlet{water}{cyan!25} 

\tikzset{
    faucet/.pic={ 
        \fill[water](-0.25,-0.25) rectangle (0.25,0.25);
        \draw[line width=1pt](-0.25,-0.25)--(0.25,-0.25) (-0.25,0.25)--(0.25,0.25);
    },
    faucet2/.pic={ 
        \fill[water](-0.25,-0.25) rectangle (0.25,0.75);
        \draw[line width=1pt](-0.25,-0.25)--(-0.25,0.75) (0.25,-0.25)--(0.25,0.75);
    },    
    faucet3/.pic={ 
        \fill[water](-0.25,-0.25) rectangle (0.25,1.21);
        \fill[water](-0.24,1.21) arc (135:225:0.35);
        \fill (0,1) circle (2pt);
        \draw[line width=1pt](-0.25,-0.25)--(-0.25,0.75) (0.25,-0.25)--(0.25,0.75) (0.25,0.75)--(-0.25,1.21);
        \draw[line width=1pt]  (-0.25,1.21) arc(135:225:0.35);
    },    
    myarrow/.tip={Stealth[scale=1.5]}, 
    surface water/.style= 
    {decoration={random steps,segment length=1mm,amplitude=0.5mm}, decorate}
}

\begin{document}

\begin{frontmatter}

\title{{\bf Numerical integration rules with improved accuracy close to singularities}}\tnotetext[label1]{Dr. Z. Li has been partially supported by a Simon's grant 633724, the rest of the authors by the project 20928/PI/18 (Proyecto financiado por la Comunidad Aut\'onoma de la Regi\'on de Murcia a trav\'es de la convocatoria de Ayudas a proyectos para el desarrollo de investigaci\'on cient\'ifica y t\'ecnica por grupos competitivos, incluida en el Programa Regional de Fomento de la Investigaci\'on Cient\'ifica y T\'ecnica (Plan de Actuaci\'on 2018) de la Fundaci\'on S\'eneca-Agencia de Ciencia y Tecnolog\'ia de la Regi\'on de Murcia) and by the Spanish national research project PID2019-108336GB-I00.}

\author[UPCT]{Sergio Amat}
\ead{sergio.amat@upct.es}
\author[NCSU]{Zhilin Li}
\ead{zhilin@ncsu.edu}
\author[UPCT]{Juan Ruiz-\'Alvarez\corref{cor1}}
\ead{juan.ruiz@upct.es}
\author[UPCT]{Concepci\'on Solano}
\ead{conchisolalorente@gmail.es}
\author[UPCT]{Juan C. Trillo}
\ead{jc.trillo@upct.es}
\date{Received: date / Accepted: date}

\address[UPCT]{Departamento de Matem\'atica Aplicada y Estad\'istica. Universidad  Polit\'ecnica de Cartagena. Cartagena, Spain.}
\address[NCSU]{Department of Mathematics. North Carolina State University. Raleigh, North Carolina, USA.}

\cortext[cor1]{Corresponding author}

\begin{abstract}

Sometimes it is necessary to obtain a numerical integration using only discretised data. In some cases, the data contains singularities which position is known but does not coincide with a discretisation point, and the jumps in the function and its derivatives are available at these positions. The motivations of this paper is to use the previous information to obtain numerical quadrature formulas that allow approximating the integral of the discrete data over certain intervals accurately.

This work is devoted to the construction and analysis of a new nonlinear technique that allows to obtain accurate numerical integrations of any order using data that contains singularities, and when the integrand is only known at grid points. The novelty of the technique consists in the inclusion of correction terms with a closed expression that depends on the size of the jumps of the function and its derivatives at the singularities, that are supposed to be known. The addition of these terms allows recovering the accuracy of classical numerical integration formulas even close to the singularities, as these correction terms account for the error that the classical integration formulas commit up to their accuracy at smooth zones. Thus, the correction terms can be added during the integration or as post-processing, which is useful if the main calculation of the integral has been already done using classical formulas. The numerical experiments performed allow us to confirm the theoretical conclusions reached in this paper.
\end{abstract}
\begin{keyword}
Accurate numerical integration formulas\sep adaption to singularities\sep definite integration\sep adapted interpolation
\sep 65D05\sep 65D17\sep 65M06\sep 65N06.
\end{keyword}
\end{frontmatter}

\section{Introduction}

Classical integration formulas, such as the trapezoidal rule, the Simpson's rule, or the Newton-Cotes formulas, are based on the integration of interpolatory polynomials over an interval. The classical problem that arises from using such interpolatory polynomials is the loss of accuracy whenever the original data does not present enough regularity. In this article, we introduce a new method inspired by the IIM \cite{Li}, created as a high-resolution technique for the discretization of elliptic partial differential equations with interfaces. 

The problem of obtaining quadrature rules adapted to the presence of discontinuities in this context can be found in the literature \cite{GRIER2014193, Tornberg2002}, but we have not found many references about the subject. In this article, we pretend to obtain of adapted integration formulas that manage to take into account the presence of discontinuities through the addition of correction terms with closed explicit ex\-pre\-ssions. To find these correction terms, we need to know the position of the singularities plus the jumps in the function and its derivatives at the singularities. We are interested in the cases when the function that is to be integrated is given as discretised data points, and we want to use these data in order to recover an approximation of the integral of the original function. In this case, the new technique can be used as a post-processing that makes explicit use of the position of the singularity and the jumps in the function and its derivatives at the singularity. Only with this information, we can compute the correction terms that allow increasing the accuracy close to the discontinuity. Our aim is to show that, through this new technique, it is possible to reach the maximum theoretical accuracy in terms of the length of the stencil.

The present work is organized as follows: Section \ref{cuentas} describes how to obtain correction terms for the trapezoid rule and Simpson's rule. Section \ref{NC} presents a generalization for Newton-Cotes formulas. Section \ref{CF} presents expressions of the correction terms for commonly used Newton-Cotes Formulas. Section \ref{experimentos} presents some numerical experiments that endorse the theoretical results. 
Finally, Section \ref{conclusions} presents some conclusions.

\section{Obtainment of adapted numerical integration formulas}\label{cuentas}
We consider the space of finite 
sequences $V$ and a uniform partition $X$ of the interval 
$[a, b]$ in $J$ subintervals, 
$$X=\{x_i\}^{J}_{i=0}, \quad x_0=a, \quad h=x_i-x_{i-1}, \quad x_J=b.$$
We will consider a piecewise smooth function $f$ discretized through the point values,
\begin{equation}\label{pv}
f_i=f(x_i), \quad f=\left\{f_i\right\}_{i=0}^{J},
\end{equation}
that, therefore, conserves the information of $f$ only at the $x_i$ nodes. We also assume that discontinuities are placed far enough from each other and that their position is known exactly or can be approximated with enough accuracy. Figure \ref{fig_disc} presents the kind of singularities that we will be dealing with in this work. We will refer to these figures along the article. 
From these considerations, we can directly proceed to obtain the correction terms and error formulas for these cases. Let us start with the trapezoidal rule.

\subsection{Error formula for the corrected trapezoid rule}\label{sec_trapecio}

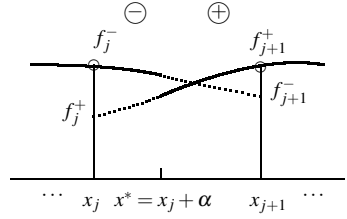
\begin{figure}[ht]
\begin{center}
\resizebox{14cm}{!} {
\begin{picture}(500,90)(-50,-10)
\put(120,-10){\line(100,0){160}}
\put(160,-10){\line(0,1){5}}
\put(192,-10){\line(0,1){5}}
\put(240,-10){\line(0,1){5}}
\put(135,-20){$\cdots$}
\put(160,55){$f^-_{j}$}
\put(235,54){$f^+_{j+1}$}
\put(260,-20){$\cdots$}

\put(155,-23){$x_{j}$}
\put(170,-23){$x^*=x_{j}+\alpha$}
\put(235,-23){$x_{j+1}$}


\put(160,-10){\line(0,1){55}}
\put(160,45){\circle{5}}
\put(240,-10){\line(0,1){54}}
\put(240,44){\circle{5}}
\linethickness{0.3mm}
\qbezier(130,45)(170,45)(192,40)
\qbezier[15](160,20)(180,25)(192,30)
\put(145,20){$f^+_{j}$}
\qbezier(192,30)(240,50)(270,45)
\qbezier[20](192,40)(210,35)(240,30)
\put(245,30){$f^-_{j+1}$}
\put(180,70){\circle{10}}
\put(176,68){$-$}
\put(220,70){\circle{10}}
\put(216,68){$+$}
\end{picture}
}
\end{center}
\caption{An example of a function with singularities (solid line) placed at a position $x^*$. We have labeled the domain to the left of the singularity as $-$ and the one to the right as $+$. We have also represented with a dashed line the prolongation of the functions through Taylor expansions at both sides of the discontinuity.}\label{fig_disc}
\end{figure}

We can consider the situation presented in Figure \ref{fig_disc}. Let us denote by $E(f)$ the error committed by the classical trapezoidal rule and by $E^*(f)$ the error by the corrected rule.
The classical trapezoid rule for a uniform grid of mesh-size $h$ and its error \cite{atkinson} at smooth zones reads,
\begin{equation}\label{trapecio}
\begin{aligned}
I(f)&=\frac{h}{2}\left(f_{j}+f_{j+1}\right), \\
E(f)&=-\frac{h^3}{12}f''(\eta),\quad \eta\in[x_{j}, x_{j+1}].
 \end{aligned}
\end{equation}
The approximation error is of order $O(h^2)$ if there is a jump in the first derivative in the interval $[x_j, x_{j+1}]$ or $O\left(h\right)$ if there is a jump discontinuity in the function. One way of rising the order of accuracy in the previous cases is to use the location of the singularity $x^*$. Let us suppose that $x^*$ is known exactly. In order to obtain the area below the curve in the interval $[x_j, x^*]$ (the area in the interval $[x^*, x_{j+1}]$ can be obtained in a similar way), we can just use the Taylor expansion of the value $f_{j+1}^+$ around $x^*$ and then change the values from the $+$ side in terms of the $-$ side using the jump relations. 
Let us use the notation,
\begin{equation}\label{ir}
\begin{aligned}
\left[f\right]&=f^{+}(x^{*}) - f^{-}(x^{*}),\\
\left[ f' \right]&=f_x^{+}(x^{*}) - f_x^{-}(x^{*}),\\
\left[ f'' \right]&=f_{xx}^{+}(x^{*}) - f_{xx}^{-}(x^{*}),\\
\left[ f''' \right]&=f_{xxx}^{+}(x^{*}) - f_{xxx}^{-}(x^{*}),\cdots
\end{aligned}
\end{equation}
for the jumps in the function and its derivatives at $x^*$. Then, using Taylor expansions at both sides of the discontinuity, the expressions for $f^-_j, f^+_j, f^-_{j+1}$ and $f^+_{j+1}$ can be written as,
\begin{equation}\label{fjm1_corr}
\begin{aligned}
f^-(x_j)&=f^-_j=f^-(x^*)-f^-_x(x^*) \alpha+O(h^2),\\
f^+(x_j)&=f^+_j=f^+(x^*)-f^+_x(x^*) \alpha+O(h^2),\\
f^-(x_{j+1})&=f^-_{j+1}=f^-(x^*)+f^-_x(x^*) (h-\alpha) +O(h^2),\\
f^+(x_{j+1})&=f^+_{j+1}=f^+(x^*)+f^+_x(x^*) (h-\alpha) +O(h^2),
\end{aligned}
\end{equation}
and subtracting we obtain,
\begin{equation}\label{fjm1_corr2}
\begin{aligned}
f_j^+&=f_j^-+[f]-[f']\alpha+O(h^2),\\
f_{j+1}^+&=f_{j+1}^-+[f]+[f'](h-\alpha)+O(h^2).
\end{aligned}
\end{equation}

Now, let us try to analyze the error formula for the corrected trapezoid rule.
We will use the following lemma, which proof is a classical result and can be found, for example, on page 143 of \cite{atkinson},
\begin{lemma}\label{lema1}
Let $t$ be a real number, different from the nodes $x_0, x_1, \cdots, x_n$. Being $n$ the degree, the polynomial interpolation error to $f(x)$ at $t$ is $f(t)-p_n(t)=(t-x_0)\cdots(t-x_n)f[x_0,\cdots, x_n,t]$, where $f[x_0,\cdots, x_n,t]$ denotes the $(n+1)$-th order divided difference.
\end{lemma}
If we denote by $E_{[a,b]}(f)$ the error of integration in the interval $[a,b]$, now we can state the following theorem:

\begin{theorem}\label{teo1}
{ Let $f(x)\in C^2([x_0,\, x^*] \cup [x^*, \, x_n])$ except at a point $x^*\in (x_j,x_{j+1})$. We  denote the function to the left of $x^*$ by $f^-(x)$ and to the right of $x^*$ as $f^+(x)$. If we know the following jumps in the function and its derivatives at $x^*$ and they are finite, $[f]=f^+(x^*)-f^-(x^*), [f']=f'^+(x^*)-f'^-(x^*)$, then} the subtraction of the correction term,
\begin{equation}\label{trapezoid_rule_iim}
C={\frac{\left(-h+2\alpha \right)}{2}} [f]
+{\frac { \left( h{\alpha}-\alpha^{2} \right) }{2}}[f'],
\end{equation}
to the trapezoid numerical integration formula in the interval $[x_j,x_{j+1}]$ that contains the singularity assures that the error is equal to,
\begin{equation}\label{errorteo1}
E^*(f)+C=E_{[x_j,x^*]}(f)+E_{[x^*,x_{j+1}]}(f)+O(h^4),
\end{equation}
with $$E_{[x_j,x^*]}(f)=-\frac{1}{12}\left(\alpha^3 f^-_{xx}(\eta^-)\right)+O(h^4),$$ and $$E_{[x^*,x_{j+1}]}(f)=-\frac{1}{12}\left((h-\alpha)^3f^+_{xx}(\eta^+)\right)+O(h^4),$$ and $\eta^-\in[x_{j}, x^*], \eta^+\in[x^*, x_{j+1}]$.
\end{theorem}

\begin{proof}
At the $-$ part of the interval we will denote $$E(f)_{[x_j,x^*]}=\int_{x_{j}}^{x_{j}+\alpha}(f^-(x)-p(x))\ dx,$$
where $p(x)$ is the polynomial of degree 1 taking the values $f_j^-$ and $f^+_{j+1}$ at the interval endpoints $x_j$ and $x_{j+1}$, respectively.
We write this error using the Lagrange's form of the polynomial and take into account that there is a singularity at $x^*=x_j+\alpha$, so we can use the expressions in (\ref{fjm1_corr}),
\begin{equation}\label{p}
\begin{aligned}
p(x)&=\frac{x-x_{j+1}}{x_j-x_{j+1}}f^-_j+\frac{x-x_j}{x_{j+1}-x_j}f^+_{j+1}\\
&=\frac{x-x_{j+1}}{x_j-x_{j+1}}f^-_j+\frac{x-x_j}{x_{j+1}-x_j}f^-_{j+1}+\frac{x-x_j}{x_{j+1}-x_j}\left(\left[f\right]+\left[f'\right] (h-\alpha) +O(h^2)\right)\\
&=p^-(x)+\frac{x-x_j}{x_{j+1}-x_j}\left(\left[f\right]+\left[f'\right] (h-\alpha) +\frac{[f'']}{2}(h-\alpha)^2 +O(h^3)\right).
\end{aligned}
\end{equation}
Then, using (\ref{p}) and denoting by $p^-(x)$ to the piecewise polynomial to the left of the discontinuity, the error can be expressed as,
\begin{equation}
\begin{aligned}
E^{*-}(f)&=\int_{x_{j}}^{x_{j}+\alpha}(f^-(x)-p^-(x)) \ dx=-\frac{1}{12}\alpha^3 f^-_{xx}(\eta^-)\\
&=\int_{x_{j}}^{x_{j}+\alpha} (f^-(x)-p(x))\ dx+\int_{x_{j}}^{x_{j}+\alpha} \frac{x-x_j}{x_{j+1}-x_j}\left(\left[f\right]+\left[f'\right] (h-\alpha)+\frac{\left[f''\right]}{2} (h-\alpha)^2\right)\ dx +O(h^4)\\
&=\int_{x_{j}}^{x_{j}+\alpha} (f^-(x)-p(x))\ dx+\frac {1}{2h}\left(\alpha^{2}[f]+\alpha^2(h-\alpha) [f']+\frac{\alpha^2}{2}(h-\alpha)^2 [f'']\right)+O(h^4)\\
&=E(f)_{[x_j,x^*]}+C^-+\frac{\alpha^2}{4h}(h-\alpha)^2 [f'']+O(h^4),
\end{aligned}
\end{equation}
with $\eta^-\in [x_j, x^*]$, where we have used the error for the classical trapezoid rule. So we have that in the interval $[x_j, x^*]$ the error is,
\begin{equation}\label{errortrap-}
\begin{aligned}
E^{*-}(f)&=E(f)_{[x_j,x^*]}+C^-+\frac{\alpha^2}{4h}(h-\alpha)^2 [f'']+O(h^4)=-\frac{1}{12}\alpha^3 f^-_{xx}(\eta^-),\quad \textrm{with } \eta^-\in [x_j, x^*],\\
C^-&={\frac {{\alpha}^{2}}{2h}}[f]+{\frac {\alpha^2(h-\alpha) }{2h}}[f'].
\end{aligned}
\end{equation}
Replicating the process for the interval $[x^*,x_{j+1}]$, but this time expressing the quantities from the $-$ side in terms of the $+$ side (or just by symmetry), we obtain that,
\begin{equation}\label{errortrap+}
\begin{aligned}
E^{*+}(f)&=E(f)_{[x^*,x_{j+1}]}+C^+-\frac{\alpha^2}{4h}(h-\alpha)^2[f'']+O(h^4)=-\frac{1}{12}(h-\alpha)^3 f^+_{xx}(\eta^+),\quad \textrm{with } \eta^+\in [x^*, x_{j+1}],\\
C^+&=-\left({\frac { \left( h-{\alpha} \right)^2}{2h}} [f]-{\frac { \left( h-{\alpha}
 \right)^2\alpha }{2h}}[f']\right).
\end{aligned}
\end{equation}
Adding the errors obtained in both intervals, as expressed in (\ref{errortrap-}) and (\ref{errortrap+}), it is easy to check that the terms of the error that are $O(h^3)$ disappear and we get, 
$$E^{*}(f)=E^{*-}(f)+E^{*+}(f)=E(f)_{[x^*,x_{j+1}]}+C^-+E(f)_{[x_j,x^*]}+C^+=E(f)_{[x^*,x_{j+1}]}+E(f)_{[x_j,x^*]}+C+O(h^4),$$
where,
$$C=C^++C^-={\frac{\left(-h+2\alpha \right)}{2}} [f]
+{\frac { \left( h{\alpha}-\alpha^{2} \right) }{2}}[f'],$$
that allows us to finish the proof. 
\end{proof}

\subsection{Correction terms and error formula for the corrected Simpson's $\frac{1}{3}$ rule}\label{sec_simpson}

\begin{figure}[ht]
\begin{minipage}{.2\textwidth}
\begin{center}
\resizebox{13cm}{!} {
\begin{picture}(500,120)(100,-20)
\linethickness{0.3mm}
\put(120,-10){\line(100,0){230}}
\put(192,-10){\line(0,1){5}}
\put(240,-10){\line(0,1){5}}
\put(320,-10){\line(0,1){5}}
\put(120,-20){$\cdots$}
\put(160,55){$f_{j-1}^-$}
\put(235,54){$f_{j}^+$}
\put(315,61){$f_{j+1}^+$}
\put(340,-20){$\cdots$}

\put(155,-23){$x_{j-1}$}
\put(170,3){$x^*=x_{j-1}+\alpha$}
\put(235,-23){$x_{j}$}
\put(315,-23){$x_{j+1}$}


\put(160,-10){\line(0,1){55}}
\put(160,45){\circle{5}}
\put(240,-10){\line(0,1){54}}
\put(240,44){\circle{5}}
\put(320,-10){\line(0,1){61}}
\put(320,51){\circle{5}}

\qbezier(130,50)(160,50)(192,40)
\qbezier[60](192,40)(250,20)(320,20)
\put(325,17){$f^-_{j+1}$}
\put(240,15){$f^-_{j}$}
\qbezier(192,30)(240,50)(340,50)
\qbezier[20](160,20)(180,25)(192,30)
\put(140,20){$f^+_{j-1}$}
\put(180,70){\circle{10}}
\put(176,68){$-$}
\put(220,70){\circle{10}}
\put(216,68){$+$}
\end{picture}
}
\end{center}
\end{minipage}
\begin{minipage}{.2\textwidth}
\resizebox{13cm}{!} {
\begin{picture}(500,120)(-150,-20)
\linethickness{0.3mm}
\put(50,-10){\line(1,0){230}}
\put(80,-10){\line(0,1){5}}
\put(160,-10){\line(0,1){5}}
\put(192,-10){\line(0,1){5}}
\put(240,-10){\line(0,1){5}}
\put(55,-20){$\cdots$}
\put(80,77){$f^-_{j-1}$}
\put(160,55){$f^-_{j}$}
\put(235,54){$f^+_{j+1}$}
\put(260,-20){$\cdots$}

\put(75,-23){$x_{j-1}$}
\put(155,-23){$x_{j}$}
\put(170,3){$x^*=x_{j+1}-\alpha$}
\put(235,-23){$x_{j+1}$}


\put(80,-10){\line(0,1){77}}
\put(80,67){\circle{5}}
\put(160,-10){\line(0,1){55}}
\put(160,45){\circle{5}}
\put(240,-10){\line(0,1){54}}
\put(240,44){\circle{5}}

\qbezier(60,70)(160,50)(192,35)
\qbezier[20](192,35)(210,27)(240,20)
\put(245,20){$f^-_{j+1}$}

\qbezier(192,25)(240,45)(270,48)
\qbezier[40](80,5)(135,7)(192,25)
\put(140,20){$f^+_{j}$}
\put(60,5){$f^+_{j-1}$}
\put(180,70){\circle{10}}
\put(176,68){$-$}
\put(220,70){\circle{10}}
\put(216,68){$+$}
\end{picture}
}
\end{minipage}
\caption{Two examples of functions with singularities (solid line) placed in different intervals at a position $x^*$. We have labeled the domain to the left of the singularity as $-$ and the one to the right as $+$. We have also represented with a dashed line the prolongation of the functions through Taylor expansions at both sides of the discontinuity.}\label{fig_disc2}
\end{figure}
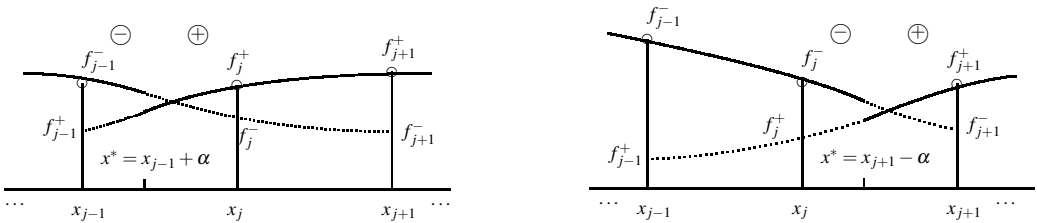

In this section we will proceed to analyze how to adapt Simpon's rule following the same process that we used to adapt the trapezoidal rule in the previous Subsection. Simpson's rule is obtained by inte\-gra\-ting a parabola in the corresponding interval. In this case we need to enlarge the stencil and we will need to use the three data values $(f_{j-1}, f_j, f_{j+1})$, placed at the positions $(x_{j-1}, x_j, x_{j+1})$ in order to build the parabola. In this occasion we must consider two cases: when the discontinuity is in the interval $[x_{j-1}, x_j]$ or in the interval $[x_j, x_{j+1}]$, as shown in the plots of Figure \ref{fig_disc2}. The classical Simpson's $\frac{1}{3}$ rule for a uniform grid of mesh-size $h$ and its error \cite{atkinson} at smooth zones reads,
\begin{equation}\label{simps}
\begin{aligned}
I(f)&=\frac{h}{3}\left(f_{j}+4f_{j+1}+f_{j+2}\right), \\
E(f)&=-\frac{h^5}{90}f^{(4)}(\eta),\quad \eta\in[x_{j}, x_{j+2}].
 \end{aligned}
\end{equation}

Now we can state the following theorem.
\begin{theorem}\label{teo2}
{  Let $f(x)\in C^3([x_0,\, x^*] \cup [x^*, \, x_n])$ except at a point $x^*\in (x_j,x_{j+1})$. We denote the function to the left of $x^*$ by $f^-(x)$ and to the right of $x^*$ as $f^+(x)$. If we know the following jumps in the function and its derivatives at $x^*$ and they are finite,  $[f]=f^+(x^*)-f^-(x^*), [f']=f'^+(x^*)-f'^-(x^*), [f'']=f''^+(x^*)-f''^-(x^*)$, then} the subtraction of the correction term,
\begin{equation}\label{errorteo2.1}
\scalemath{0.9}{
\begin{aligned}
C&=\gamma\left( \alpha-\frac{h}{3} \right) [f]+\frac{\alpha}{6} \left( 3
\alpha-2h \right) [f']+\gamma\frac{{\alpha}^{2}}{6} \left( \alpha-h
 \right) [f''],
\end{aligned}
}
\end{equation}
to the Simpson's numerical integration formula, with $\gamma=1$, if the singularity is placed at an odd interval, and $\gamma=-1$,
if the singularity is placed at an even interval, assures that the error is equal to,
\begin{equation}\label{errorsimps}
\scalemath{0.9}{
\begin{aligned}
E(f)+C&=\frac{{\alpha}^{2}}{36} \left( 3{\alpha}^{2}+6{h}
^{2}-8h\alpha \right) [f''']+\frac{f^+_{xxxx}(\eta_1)}{24}\left(\frac{3{h}^{2}{\alpha}^{2}}{2}-h{\alpha}^{3}-\frac{{h}^{4}}{4}\right)+\frac{f^+_{xxx}(\eta_2)}{6}\left(-\frac{{\alpha}^{4}}{4}+\frac{{h}^{2}{\alpha}^{2}}{2}\right)\\
&+\frac{f^-_{xxx}(\eta_3) }{24}\left(\frac{\alpha^{4}}{4}-h{\alpha}^{3}+{h}^{2}{\alpha}^{2}\right)+O(h^5),
\end{aligned}
}
\end{equation}
with $\eta_1\in[x_{j-1}+\alpha,x_{j+1}-\alpha]$. If the discontinuity falls at an odd interval, then $\eta_2\in[x_{j+1}-\alpha,x_{j+1}], \eta_3\in[x_{j-1},x_{j-1}+\alpha]$. If the discontinuity falls at an even interval, the case is symmetric and $\eta_2\in[x_{j-1},x_{j-1}+\alpha], \eta_3\in[x_{j+1}-\alpha,x_{j+1}]$.
\end{theorem}
\begin{proof}

\begin{itemize}
\item We start by the case when the discontinuity is placed in the interval $[x_{j-1}, x_j]$. 
\begin{enumerate}
\item As in the trapezoidal rule, we know that for the $+$ part of the integral,
 $$E^{*+}(f)=\int_{x_{j-1}+\alpha}^{x_{j+1}}(f^+(x)-p^+(x)) \ dx.$$
The interpolating polynomial $p(x)$ in the Lagrange form is,
\begin{equation}\label{pol3}
\scalemath{0.9}{
\begin{aligned}
p(x)&=\frac{(x-x_{j})(x-x_{j+1})}{(x_{j-1}-x_{j})(x_{j-1}-x_{j+1})}f^-_{j-1}+\frac{(x-x_{j-1})(x-x_{j+1})}{(x_{j}-x_{j-1})(x_{j}-x_{j+1})}f^+_{j}+\frac{(x-x_{j-1})(x-x_{j})}{(x_{j+1}-x_{j-1})(x_{j+1}-x_{j})}f^+_{j+1}.
\end{aligned}
}
\end{equation}
Proceeding in the same way as we did in (\ref{fjm1_corr}) for the trapezoid rule, we can use the expression of $f^-_{j-1}$ in terms of the quantities from the $+$ side to write,
\begin{equation}\label{fjm1_corr3}
\scalemath{0.9}{
\begin{aligned}
f^-_{j-1}
&=f^+_{j-1}-\left[f\right]+[f'] \alpha-[f''] \frac{\alpha^2}{2}+[f''']\frac{\alpha^3}{6}+O(h^4).\\
\end{aligned}
}
\end{equation}
Now we can write,
\begin{equation}
\scalemath{0.9}{
\begin{aligned}
p(x)&=\frac{(x-x_{j})(x-x_{j+1})}{(x_{j-1}-x_{j})(x_{j-1}-x_{j+1})}\left(f^+_{j-1}-\left[f\right]+[f'] \alpha-[f''] \frac{\alpha^2}{2}+[f''']\frac{\alpha^3}{6}\right)
+\frac{(x-x_{j-1})(x-x_{j+1})}{(x_{j}-x_{j-1})(x_{j}-x_{j+1})}f^+_{j}\\
&+\frac{(x-x_{j-1})(x-x_{j})}{(x_{j+1}-x_{j-1})(x_{j+1}-x_{j})}f^+_{j+1}\\
&=p^+(x)+\frac{(x-x_{j})(x-x_{j+1})}{(x_{j-1}-x_{j})(x_{j-1}-x_{j+1})}\left(-\left[f\right]+[f'] \alpha-[f''] \frac{\alpha^2}{2}+[f''']\frac{\alpha^3}{6}\right)+O(h^4).\\
\end{aligned}
}
\end{equation}
Then, the error for the integral at the $+$ side in the interval $[x^*, x_{j+1}]$, as shown in Figure \ref{fig_disc} to the left, can be expressed as,
\begin{equation}\label{int_mas0}
\begin{aligned}
E^{*+}(f)&=\int_{x_{j-1}+\alpha}^{x_{j+1}}(f^+(x)-p^+(x)) \ dx=\int_{x_{j-1}+\alpha}^{x_{j+1}}(f^+(x)-p(x)) \ dx\\
&+\int_{x_{j-1}+\alpha}^{x_{j+1}} \frac{(x-x_{j})(x-x_{j+1})}{(x_{j-1}-x_{j})(x_{j-1}-x_{j+1})}\left(-\left[f\right]+[f'] \alpha-[f''] \frac{\alpha^2}{2}+[f''']\frac{\alpha^3}{6}\right)\ dx+O(h^5)\\
&=
E(f)_{[x^*,x_{j+1}]}\\
&-{\frac {1}{72}}\,{\frac {- \left( 6\,[f]-6\,\alpha\,[f']+3\,{\alpha}^{2}[f'']-\,{\alpha}^{3}[f''']
 \right) \left( -4\,{h}^{3}+2\,{\alpha}^{3}-9\,h{\alpha}^{2}+12\,{h}^{2}\alpha \right) }{{h}^{2}}}+O(h^5)\\
&=E(f)_{[x^*,x_{j+1}]}+C^++O(h^5),
\end{aligned}
\end{equation}
and we also have that,
\begin{equation}\label{int_mas}
\begin{aligned}
E^{*+}(f)&=\int_{x_{j-1}+\alpha}^{x_{j+1}} (x-x_{j-1})(x-x_{j})(x-x_{j+1})f^+[x_{j-1}, x_{j}, x_{j+1},x]\ dx.
\end{aligned}
\end{equation}
The polynomial in the integrand of (\ref{int_mas}) changes the sign in the interval $(x_{j-1}+\alpha,x_{j+1})$. Thus, we can not use the integral mean value theorem. Instead, we can define the function $$w(x)=\int_{x_{j-1}+\alpha}^{x} (x-x_{j-1})(x-x_{j})(x-x_{j+1})\ dx,$$
that satisfies, $w(x_{j-1}+\alpha)=0$, and $w(x)>0$ for $x\in(x_{j-1}+\alpha,x_{j+1}-\alpha)$ and $w(x)<0$ for $x\in(x_{j+1}-\alpha, x_{j+1})$. Then, we can divide the integral in two parts,
\begin{equation}\label{2partes}
\begin{aligned}
E^{*+}(f)&=\int_{x_{j-1}+\alpha}^{x_{j+1}} w'(x)f^+[x_{j-1}, x_{j}, x_{j+1},x]\ dx=\int_{x_{j-1}+\alpha}^{x_{j+1}-\alpha} w'(x)f^+[x_{j-1}, x_{j}, x_{j+1},x]\ dx\\
&+\int_{x_{j+1}-\alpha}^{x_{j+1}} w'(x)f^+[x_{j-1}, x_{j}, x_{j+1},x]\ dx.
\end{aligned}
\end{equation}
Integrating by parts the first integral,
\begin{equation*}
\begin{aligned}
\int_{x_{j-1}+\alpha}^{x_{j+1}-\alpha} w'(x) f^+[x_{j-1}, x_{j}, x_{j+1},x]\ dx
&=\left[w(x) f^+[x_{j-1}, x_{j}, x_{j+1},x]\right]_{x_{j-1}+\alpha}^{x_{j+1}-\alpha}\\&-\int_{x_{j-1}+\alpha}^{x_{j+1}-\alpha} w(x) \frac{d}{\ dx}f^+[x_{j-1}, x_{j}, x_{j+1},x]\ dx.
\end{aligned}
\end{equation*}
Using now that $w(x_{j+1}-\alpha)=0$, due to the symmetry of the polynomial that appears in the integrand of $w(x)$ in a uniform grid, that (see 3.2.17 page 147 of Atkinson)
\begin{equation}\label{der_difdiv}
\frac{d}{dx}f^+[x_{j-1}, x_{j}, x_{j+1},x]=f^+[x_{j-1}, x_{j}, x_{j+1},x,x],
\end{equation}
 and the integral mean value theorem, we get,
\begin{equation}
\begin{aligned}
&-\int_{x_{j-1}+\alpha}^{x_{j+1}-\alpha} w(x) f^+[x_{j-1}, x_{j}, x_{j+1},x,x]\ dx=-f^+[x_{j-1}, x_{j}, x_{j+1},\xi_1,\xi_1]\int_{x_{j-1}+\alpha}^{x_{j+1}-\alpha} w(x) \ dx=\\
&-f^+[x_{j-1}, x_{j}, x_{j+1},\xi_1,\xi_1]\left(-\frac{3{h}^{2}{\alpha}^{2}}{2}+h{\alpha}^{3}+\frac{{h}^{4}}{4}\right)=\frac{f^{+}_{xxxx}(\eta_1)}{24}\left(\frac{3{h}^{2}{\alpha}^{2}}{2}-h{\alpha}^{3}-\frac{{h}^{4}}{4}\right),
\end{aligned}
\end{equation}
for some $\xi_1, \eta_1\in[x_{j-1}+\alpha,x_{j+1}-\alpha]$. For the second integral in (\ref{2partes}), $w'(x)$ does not change the sign in $[x_{j+1}-\alpha,x_{j+1}]$ so we can apply the integral mean value theorem,
\begin{equation*}
\begin{aligned}
\int_{x_{j+1}-\alpha}^{x_{j+1}} w'(x)f^+[x_{j-1}, x_{j}, x_{j+1},x]\ dx&=f^+[x_{j-1}, x_{j}, x_{j+1},\xi_2]\int_{x_{j+1}-\alpha}^{x_{j+1}} w'(x)\ dx\\
&=\frac{f^+_{xxx}(\eta_2)}{6}\left(-\frac{{\alpha}^{4}}{4}+\frac{{h}^{2}{\alpha}^{2}}{2}\right),
\end{aligned}
\end{equation*}
for some $\xi_2, \eta_2\in[x_{j+1}-\alpha,x_{j+1}]$. Thus,
\begin{equation*}
\begin{aligned}
\int_{x_{j-1}+\alpha}^{x_{j+1}} w'(x)f^+[x_{j-1}, x_{j}, x_{j+1},x]\ dx&=\frac{f^+_{xxxx}(\eta_1)}{24}\left(\frac{3{h}^{2}{\alpha}^{2}}{2}-h{\alpha}^{3}-\frac{{h}^{4}}{4}\right)\\&+\frac{f^+_{xxx}(\eta_2)}{6}\left(-\frac{{\alpha}^{4}}{4}
+\frac{{h}^{2}{\alpha}^{2}}{2}\right).
\end{aligned}
\end{equation*}
So, from (\ref{int_mas0}) we get that the corrected error for the integral in the $+$ side of the left plot of Figure \ref{fig_disc2} is,
\begin{equation}\label{error+}
\begin{aligned}
E(f)_{[x^*,x_{j+1}]}+C^++O(h^5)=E^{*+}(f)&=\frac{f^+_{xxxx}(\eta_1)}{24}\left(\frac{3{h}^{2}{\alpha}^{2}}{2}-h{\alpha}^{3}-\frac{{h}^{4}}{4}\right)\\
&+\frac{f^+_{xxx}(\eta_2)}{6}\left(-\frac{{\alpha}^{4}}{4}+\frac{{h}^{2}{\alpha}^{2}}{2}\right),
\end{aligned}
\end{equation}
with $\xi_2, \eta_2\in[x_{j+1}-\alpha,x_{j+1}]$ and $\xi_1, \eta_1\in[x_{j-1}+\alpha,x_{j+1}-\alpha]$.
\item For the integral in the $-$ side of the left plot of Figure \ref{fig_disc2}, we want to obtain the error $$E^{*-}(f)=\int^{x_{j-1}+\alpha}_{x_{j-1}}(f^-(x)-p^-(x)) \ dx.$$
From (\ref{pol3}) we can express the quantities from the $+$ side in terms of the $-$ side using the jump conditions in (\ref{ir}), as we did before,
\begin{equation}
\scalemath{0.9}{
\begin{aligned}
p(x)&=\frac{(x-x_{j})(x-x_{j+1})}{(x_{j-1}-x_{j})(x_{j-1}-x_{j+1})}f^-_{j-1}\\
&+\frac{(x-x_{j-1})(x-x_{j+1})}{(x_{j}-x_{j-1})(x_{j}-x_{j+1})}\left(f^-_{j}+\left[f\right]+[f'] (h-\alpha)+[f''] \frac{(h-\alpha)^2}{2}+[f''']\frac{(h-\alpha)^3}{6}\right)\\
&+\frac{(x-x_{j-1})(x-x_{j})}{(x_{j+1}-x_{j-1})(x_{j+1}-x_{j})}\left(f^-_{j+1}+\left[f\right]+[f'] (h+\alpha)+[f''] \frac{(h+\alpha)^2}{2}+[f''']\frac{(h+\alpha)^3}{6}\right)+O(h^4)\\
&=p^-(x)+\frac{(x-x_{j-1})(x-x_{j+1})}{(x_{j}-x_{j-1})(x_{j}-x_{j+1})}\left(\left[f\right]+[f'] (h-\alpha)+[f''] \frac{(h-\alpha)^2}{2}+[f''']\frac{(h-\alpha)^3}{6}\right)\\
&+\frac{(x-x_{j-1})(x-x_{j})}{(x_{j+1}-x_{j-1})(x_{j+1}-x_{j})}\left(\left[f\right]+[f'] (h+\alpha)+[f''] \frac{(h+\alpha)^2}{2}+[f''']\frac{(h+\alpha)^3}{6}\right)+O(h^4).
\end{aligned}
}
\end{equation}

Now, the error for the integral on the $-$ side, as shown in Figure \ref{fig_disc2} to the left, can be expressed as,
\begin{equation}\label{int_menos}
\scalemath{0.85}{
\begin{aligned}
E^-(f)&=\int_{x_{j-1}}^{x_{j-1}+\alpha}(f^-(x)-p^-(x)) \ dx=\int_{x_{j-1}}^{x_{j-1}+\alpha}(f^-(x)-p(x)) \ dx\\
&+\int_{x_{j-1}}^{x_{j-1}+\alpha}\frac{(x-x_{j-1})(x-x_{j+1})}{(x_{j}-x_{j-1})(x_{j}-x_{j+1})}\left(\left[f\right]+[f'] (h-\alpha)+[f''] \frac{(h-\alpha)^2}{2}+[f''']\frac{(h-\alpha)^3}{6}\right)\ dx\\
&+\int_{x_{j-1}}^{x_{j-1}+\alpha}\frac{(x-x_{j-1})(x-x_{j})}{(x_{j+1}-x_{j-1})(x_{j+1}-x_{j})}\left(\left[f\right]+[f'] (2h-\alpha)+[f''] \frac{(2h-\alpha)^2}{2}+[f''']\frac{(2h-\alpha)^3}{6}\right)\ dx+O(h^5)\\
&=
E(f)-{\frac {1}{72}}\Bigg({\frac {{\alpha}^{2} \left( -54\,h+12\,\alpha
 \right) [f]}{{h}^{2}}}+\,{\frac {{\alpha}^{2}
 \left( 54\,h\alpha-12\,{\alpha}^{2}-36\,{h}^{2} \right) [f']
}{{h}^{2}}}+\,{\frac {{\alpha}^{2} \left( 6\,{\alpha}
^{3}+24\,{h}^{2}\alpha-27\,h{\alpha}^{2} \right) [f'']}{{h}^{2
}}}\\
&+\,{\frac {{\alpha}^{2} \left( -2\,{\alpha}^{4}+9
\,h{\alpha}^{3}-12\,{h}^{3}\alpha-6\,{h}^{2}{\alpha}^{2}+12\,{h}^{4}
 \right) [f''']}{{h}^{2}}}\Bigg)+O(h^5)\\
&=E(f)_{(x_{j-1},x^*)}+C^-+O(h^5)=\int_{x_{j-1}}^{x_{j-1}+\alpha} (x-x_{j-1})(x-x_{j})(x-x_{j+1})f^-[x_{j-1}, x_{j}, x_{j+1},x]\ dx.
\end{aligned}
}
\end{equation}

It is not difficult to see that the polynomial in the integrand does not change the sign in the interval $(x_{j-1},x_{j-1}+\alpha)$. Thus, using the integral mean value theorem
\begin{equation}
\begin{aligned}
E^{*-}(f)&=\int_{x_{j-1}}^{x_{j-1}+\alpha} (x-x_{j-1})(x-x_{j})(x-x_{j+1})f^-[x_{j-1}, x_{j}, x_{j+1},x]\ dx\\
&=f^-[x_{j-1}, x_{j}, x_{j+1},\xi_3]\int_{x_{j-1}}^{x_{j-1}+\alpha} (x-x_{j-1})(x-x_{j})(x-x_{j+1})\ dx\\
&=\frac{f^-_{xxx}(\eta_3) }{24}\left(\frac{\alpha^{4}}{4}-h{\alpha}^{3}+{h}^{2}{\alpha}^{2}\right),
\end{aligned}
\end{equation}
for some $\xi_3, \eta_3\in[x_{j-1},x_{{j-1}}+\alpha]$.
So, from (\ref{int_menos}) we get that the corrected error for the left part of the integral is,
\begin{equation}\label{error-}
E(f)_{(x_{j-1},x^*)}+C^-+O(h^5)=E^{*-}(f)=\frac{f^-_{xxx}(\eta_3) }{24}\left(\frac{\alpha^{4}}{4}-h{\alpha}^{3}+{h}^{2}{\alpha}^{2}\right),
\end{equation}
for some $\xi_3, \eta_3\in[x_{j+1}-\alpha,x_{j+1}]$.
\end{enumerate}

Adding the error terms $C^+$ and $C^-$ obtained in (\ref{int_mas0}) and (\ref{int_menos}), we obtain 
\begin{equation}\label{cterm0}
\scalemath{0.9}{
\begin{aligned}
C^++C^-&=-\left(-\left( \alpha-\frac{h}{3} \right) [f]+\frac{\alpha}{6} \left( 3
\alpha-2h \right) [f']-\frac{{\alpha}^{2}}{6} \left( \alpha-h
 \right) [f'']+\frac{{\alpha}^{2}}{36} \left( 3{\alpha}^{2}+6{h}
^{2}-8h\alpha \right) [f''']\right).
\end{aligned}
}
\end{equation}
Let us denote the terms up to $O(h^3)$ by,
\begin{equation}\label{cterm}
\scalemath{0.9}{
\begin{aligned}
C&=-\left(-\left( \alpha-\frac{h}{3} \right) [f]+\frac{\alpha}{6} \left( 3
\alpha-2h \right) [f']-\frac{{\alpha}^{2}}{6} \left( \alpha-h
 \right) [f'']\right).
\end{aligned}
}
\end{equation}

Adding now the errors in the intervals $[x^*, x_{j+1}]$ and $[x_{j-1},x^*]$ as expressed respectively in (\ref{error+}) and (\ref{error-}), and denoting again $$E(f)=E(f)_{[x_{j-1},x^*]}+E(f)_{[x^*,x_{j+1}]},$$ we obtain from (\ref{cterm}) and (\ref{cterm0}),
\begin{equation}\label{errort}
\scalemath{0.85}{
\begin{aligned}
E^*(f)=E^{*+}(f)+E^{*-}(f)=E(f)+C+O(h^5)&=\frac{{\alpha}^{2}}{36} \left( 3{\alpha}^{2}+6{h}
^{2}-8h\alpha \right) [f''']+\frac{f^+_{xxxx}(\eta_1)}{24}\left(\frac{3{h}^{2}{\alpha}^{2}}{2}-h{\alpha}^{3}-\frac{{h}^{4}}{4}\right)\\
&+\frac{f^+_{xxx}(\eta_2)}{6}\left(-\frac{{\alpha}^{4}}{4}+\frac{{h}^{2}{\alpha}^{2}}{2}\right)+\frac{f^-_{xxx}(\eta_3) }{24}\left(\frac{\alpha^{4}}{4}-h{\alpha}^{3}+{h}^{2}{\alpha}^{2}\right),
\end{aligned}
}
\end{equation}
with $\eta_1\in[x_{j-1}+\alpha,x_{j+1}-\alpha], \eta_2\in[x_{j+1}-\alpha,x_{j+1}], \eta_3\in[x_{j-1},x_{j-1}+\alpha].$
\item If the singularity is placed in the interval $(x_{j}, x_{j+1})$ at a distance $\alpha$ from $x_{j+1}$, that is the case presented in Figure \ref{fig_disc2} to the right, the case is symmetrical and the correction term is:
\begin{equation*}
\begin{aligned}
C&=-\left(\left( \alpha-\frac{h}{3} \right) [f]+\frac{\alpha}{6} \left( 3
\alpha-2h \right) [f']+\frac{{\alpha}^{2}}{6} \left( \alpha-h
 \right) [f'']\right).
\end{aligned}
\end{equation*}
In this case the error reads,

\begin{equation}\label{errort2}
\scalemath{0.9}{
\begin{aligned}
E^*(f)=E(f)+C+O(h^5)&=\frac{{\alpha}^{2}}{36} \left( 3{\alpha}^{2}+6{h}
^{2}-8h\alpha \right) [f''']+\frac{f^+_{xxxx}(\eta_1)}{24}\left(\frac{3{h}^{2}{\alpha}^{2}}{2}-h{\alpha}^{3}-\frac{{h}^{4}}{4}\right)\\
&+\frac{f^+_{xxx}(\eta_2)}{6}\left(-\frac{{\alpha}^{4}}{4}+\frac{{h}^{2}{\alpha}^{2}}{2}\right)
+\frac{f^-_{xxx}(\eta_3) }{24}\left(\frac{\alpha^{4}}{4}-h{\alpha}^{3}+{h}^{2}{\alpha}^{2}\right),
\end{aligned}
}
\end{equation}

with $\eta_1\in[x_{j-1}+\alpha,x_{j+1}-\alpha], \eta_2\in[x_{j-1},x_{j-1}+\alpha], \eta_3\in[x_{j+1}-\alpha,x_{j+1}]$.
\end{itemize}
\end{proof}

\begin{remark}
Theorems \ref{teo1} and \ref{teo2} imply that we can use the classical composite trapezoidal rule or the composite Simpson's rule to obtain the integral over a large interval and, then, add the corres\-ponding correction terms (\ref{trapezoid_rule_iim}) or (\ref{errorteo2.1}) to obtain $O(h^2)$ or $O(h^4)$ global accuracy respectively, if singularities are present in the data. Mind that the correction terms are typically added to take into account the effect of the set of singularities, which cardinal is usually small (one dimension lower) compared with the number of points in the data. Thus, it is enough if the correction terms provide the order of the global error of the classical composite integration rule. 
The integral can be obtained through classical quadrature rules and then add the corrections as post-processing.
\end{remark}

\section{Modified Newton-Cotes integration formulas}\label{NC}

The Trapezoidal rule and the Simpson's $\frac{1}{3}$ formula, which we have analyzed in previous sections, are the first two cases of Newton-Cotes integration formulas. In what follows, we will try to obtain expressions for the errors of corrected integration formulas of any order. To do so, we present some previous lemmas that we will use afterward in the proofs.

\begin{lemma}\label{lema2}
Let $f(x)\in C^{n+1}([a,\, x^*] \cup [x^*, \, b])$ except at a point $x^*\in (a,b)$. 
We denote the function to the left of $x^*$ by $f^-(x)$ and to the right of $x^*$ as $f^+(x)$. If we know the following jumps in the function and its derivatives at $x^*$ and they are finite, $[f]=f^+(x^*)-f^-(x^*), [f']=f'^+(x^*)-f'^-(x^*),\cdots, [f^{(n)}]=f^{(n)+}(x^*)-f^{(n)-}(x^*)$, then at any node $x_i$ we can express any value of $f^{+}(x_i)$ in terms of the jumps and the continuous extension of the function from the other side of the discontinuity (see for example, Figures \ref{fig_disc}, \ref{fig_disc2}, \ref{fig_disc3}), that is:
\begin{equation}\label{exp+}
\begin{aligned}
f^+_i&=f^-_i+[f]+[f'](x_i-x^*)+\frac{1}{2}[f''](x_i-x^*)^2+\cdots+\frac{1}{n!}[f^{(n)}](x_i-x^*)^n+O(h^{n+1}).\\
\end{aligned}
\end{equation}
Isolating, we can obtain $f^-_i$ in terms of $f^+_i$.

\end{lemma}
\begin{proof}
The proof is direct using Taylor expansions.
\end{proof}

We denote by $\lfloor x\rfloor$ greatest integer less than or equal to $x$ and $\lceil x\rceil$ the least integer greater than or equal to~$x$.

\begin{lemma}\label{lema3}
We consider an interpolating polynomial of degree $n$ in the Lagrange form in the interval $[a,b]$, constructed using $n+1$ points belonging to a piecewise continuous function that contains a singularity at $x^*\in(a,b)$ and that is $n$ times piecewise continuously differentiable. We follow the same notation as before and denote the information to the left of the singularity with the $-$ symbol and to the right with the $+$ symbol. Then in the interval of interest $[a,b]$:
\begin{itemize}
\item We can express this polynomial as a continuous extension in the $-$ region of the polynomial at the $+$ region, plus additional terms as,
\begin{equation}\label{pol+}
\begin{aligned}
p_n(x)=\sum_{i=0}^{n+1}f^{+}(x_i)\prod_{j=0, j\neq i}^{j=n+1}\dfrac{x-x_j}{x_i-x_j}+Q^-(x)=p^+_n(x)+Q^-(x).
\end{aligned}
\end{equation}
If we denote by,
\begin{equation}\label{tayl}
\begin{aligned}
\tilde{f}_i&=[f]+[f'](x_i-x^*)+\frac{1}{2}[f''](x_i-x^*)^2+\cdots+\frac{1}{n!}[f^{(n)}](x_i-x^*)^n+O(h^{n+1}),
\end{aligned}
\end{equation}
then $Q^-(x)$ contains all the information of the singularity and takes the expression,
\begin{equation}\label{pol-}
\begin{aligned}
Q^-(x)=\sum_{i=0}^{\lfloor \frac{x^*-a}{h}\rfloor}\tilde{f_i} \prod_{j=0, j\neq i}^{j=n+1}\dfrac{x-x_j}{x_i-x_j},\\
\end{aligned}
\end{equation}
\item We can express this polynomial as a continuous extension in the $+$ region of the polynomial at the $-$ region, plus additional terms as,
\begin{equation*}
\begin{aligned}
p_n(x)=\sum_{i=0}^{n+1}f^{-}(x_i)\prod_{j=0, j\neq i}^{j=n+1}\dfrac{x-x_j}{x_i-x_j}+Q^+(x)=p^+_n(x)+Q^+(x).
\end{aligned}
\end{equation*}
In this case $Q^+(x)$ takes the expression,
\begin{equation*}
\begin{aligned}
Q^+(x)=\sum_{i=\lceil \frac{x^*-a}{h}\rceil}^{n+1}\tilde{f_i} \prod_{j=0, j\neq i}^{j=n+1}\dfrac{x-x_j}{x_i-x_j},\\
\end{aligned}
\end{equation*}
\end{itemize}
\end{lemma}
\begin{proof}
The proof is direct using Lemma \ref{lema2} and replacing $f_i$ in the Lagrange form of the polynomial
\begin{equation*}
\begin{aligned}
p_n(x)=\sum_{i=0}^{n+1}f_i\prod_{j=0, j\neq i}^{j=n+1}\dfrac{x-x_j}{x_i-x_j},
\end{aligned}
\end{equation*}
by the values $f^+_i$ or $f^-_i$ provided in (\ref{exp+}), depending of $f_i$ belonging to the $+$ or $-$ side.
\end{proof}

\begin{lemma}\label{lema4}
We consider the integral of the polynomial interpolation error from Lemma \ref{lema1} in the smooth interval $[x_0,x^*]$,
$$E_n=\int_{x_0}^{x^*}(f(x)-p_n(x))\ dx=\int_{x_0}^{x^*}(x-x_0)\cdots(x-x_n)f[x_0,\cdots, x_n,x]\ dx.$$
\begin{itemize}
\item If there is not a change of sign in the polynomial of the integrand in the interval $[x_0,x^*]$, the error can be written as,
\begin{equation}\label{lemma4_1}
\begin{aligned}
E_n&=\frac{f^{(n+1)}(\xi)}{(n+1)!}h^{n+2}\int_{0}^{\frac{x^*-x_0}{h}}\mu(\mu-1)\cdots(\mu-n+1)(\mu-n)\ d\mu
\end{aligned}
\end{equation}
for some $\xi\in[x_0,x_n]$.

\item If there is a change of sign in the polynomial of the integrand in the interval $[x_0,x^*]$, the error can be written as,
\begin{equation}\label{lemma4_2}
\begin{aligned}
E_n&=\int_{x_0}^{x^*}(f(x)-p_n(x))\ dx=\frac{f^{(n+1)}(\xi_1)}{(n+1)!}h^{n+2}\int_{0}^{\frac{x^*-x_0}{h}}\mu(\mu-1)\cdots(\mu-n+1)(\mu-n)\ d\mu\\
&+\frac{f^{(n+2)}(\xi_2)}{(n+2)!}h^{n+3}\int_{0}^{\frac{x^*-x_0}{h}} \mu(\mu-1)\cdots(\mu-n+1)(\mu-n)(\mu-\frac{x^*-x_0}{h})\ d\mu.
\end{aligned}
\end{equation}
for some $\xi_1,\xi_2\in[x_0,x_n]$.

\end{itemize}
\end{lemma}

\begin{proof}
First, if there is not a change of sign in the polynomial of the integrand in the smooth interval $[x_0,x^*]$, we can directly use the integral mean value theorem and the fact that,
\begin{equation}\label{dif_der}
 f[x_0,\cdots, x_n]=\frac{f^{(n)}(\xi)}{n!} \quad \textrm{ for some } \xi\in[x_0,\cdots, x_n],
\end{equation} 
to write,
\begin{equation*}
\begin{aligned}
E_n&=\int_{x_0}^{x^*}(f(x)-p_n(x))\ dx=\int_{x_0}^{x^*}(x-x_0)\cdots(x-x_n)f[x_0,\cdots, x_n,x]\ dx\\
&=\frac{f^{(n+1)}(\xi)}{(n+1)!}\int_{x_0}^{x^*} (x-x_0)\cdots(x-x_n)\ dx.
\end{aligned}
\end{equation*}
for some $\xi\in[x_0,x_n]$. Applying the change of variables $x=x_0+\mu h, 0\le \mu\le n$, we can write,
\begin{equation*}
\begin{aligned}
E_n&=\frac{f^{(n+1)}(\xi)}{(n+1)!}\int_{x_0}^{x^*} (x-x_0)\cdots(x-x_n)\ dt\\
&=\frac{f^{(n+1)}(\xi)}{(n+1)!}h^{n+2}\int_{0}^{\frac{x^*-x_0}{h}}\mu(\mu-1)\cdots(\mu-n+1)(\mu-n)\ d\mu.
\end{aligned}
\end{equation*}

Secondly, If there is a change of sign in the polynomial of the integrand in the smooth interval $[x_0,x^*]$, we can define 
\begin{equation}\label{eqw}
w(y,x)=\int_{y}^{x}(t-x_0)\cdots(t-x_n)\ dt,
\end{equation} 
that satisfies that, at smooth zones, $$w(x_0,x_0)=w(x_0,x_n)=0, \quad w(x_0,x)>0 \quad \textrm{ for } x_0<x<x_n,$$
when $n$ is even, and 
$$w(x_0,x_0)=0, \quad w(x_0,x)<0 \quad \textrm{ for } x_0<x<x_n,$$
when $n$ is odd. In \cite{IK} (page 309) there is a complete proof of these facts.

Now, we can write,
\begin{equation}\label{2partesg}
\begin{aligned}
E_n&=\int_{x_0}^{x^*} w'(x_0,x)f[x_0,\cdots, x_n,x]\ dx.
\end{aligned}
\end{equation}
Integrating by parts and using that $w(x_0,x_0)=0$,
\begin{equation}\label{int1}
\scalemath{0.9}{
\begin{aligned}
\int_{x_0}^{x^*} w'(x_0,x) f[x_0,\cdots, x_n,x]\ dx
&=\left[w(x_0,x) f[x_0,\cdots, x_n,x]\right]_{x_0}^{x^*}-\int_{x_0}^{x^*} w(x_0,x) \frac{d}{dx}f[x_0,\cdots, x_n,x]\ dx\\
&=w(x_0,x^*) f[x_0,\cdots, x_n,x^*]-\int_{x_0}^{x^*} w(x_0,x) \frac{d}{dx}f[x_0,\cdots, x_n,x]\ dx.
\end{aligned}
}
\end{equation}
Using now (\ref{dif_der}), we can write,
\begin{equation*}
\begin{aligned}
w(x_0,x^*) f[x_0,\cdots, x_n,x^*]&=\frac{f^{(n+1)}(\xi_1)}{(n+1)!}w(x_0,x^*)=\frac{f^{(n+1)}(\xi_1)}{(n+1)!}\int_{x_0}^{x^*} (t-x_0)\cdots(t-x_n)\ dt.
\end{aligned}
\end{equation*}
for some $\xi_1\in[x_0,x^*]$. Applying again the change of variables $t=x_0+\mu h, 0\le \mu\le n$, we can write,
\begin{equation*}
\begin{aligned}
&\int_{x_0}^{x^*} (t-x_0)\cdots(t-x_n)\ dt=h^{n+2}\int_{0}^{\frac{x^*-x_0}{h}}\mu(\mu-1)\cdots(\mu-n+1)(\mu-n)\ d\mu.
\end{aligned}
\end{equation*}
Thus, we have that,
\begin{equation}\label{res0}
\begin{aligned}
w(x_0,x^*) f[x_0,\cdots, x_n,x^*]=\frac{f^{(n+1)}(\xi_1)}{(n+1)!}h^{n+2}\int_{0}^{\frac{x^*-x_0}{h}}\mu(\mu-1)\cdots(\mu-n+1)(\mu-n)\ d\mu,
\end{aligned}
\end{equation}
for some $\xi_1\in[x_0,x_n]$.

For the last integral in (\ref{int1}) we can use the fact that (see 3.2.17 page 147 of Atkinson)
\begin{equation}\label{der_difdivg}
\frac{d}{dx}f[x_0,\cdots, x_n,x]=f[x_0,\cdots, x_n,x,x],
\end{equation}
the integral mean value theorem, and (\ref{dif_der}) to write
\begin{equation*}
\begin{aligned}
-\int_{x_0}^{x^*} w(x_0,x) \frac{d}{dx}f[x_0,\cdots, x_n,x]\ dx&=-\int_{x_0}^{x^*} w(x_0,x) f[x_0,\cdots, x_n,x,x]\ dx\\
&=-f[x_0,\cdots, x_n,\eta_2,\eta_2]\int_{x_0}^{x^*} w(x_0,x) \ dx\\
&=-\frac{f^{(n+2)}(\xi_2)}{(n+2)!}\int_{x_0}^{x^*} \int_{x_0}^{x} (t-x_0)\cdots(t-x_n)\ dt dx,
\end{aligned}
\end{equation*}
for some $\eta_2, \xi_2\in[x_0,x_n]$.
Now we can change the order of integration and apply the change of variables $t=x_0+\mu h, 0\le \mu\le n$:
\begin{equation}\label{res3}
\scalemath{0.9}{
\begin{aligned}
\int_{x_0}^{x^*} \int_{x^*}^{x} (t-x_0)\cdots(t-x_n)\ dtdx&=\int_{x_0}^{x^*} \int_{t}^{x^*} (t-x_0)\cdots(t-x_n)\ dxdt\\
=&\int_{x_0}^{x^*} (t-x_0)\cdots(t-x_n)(d-t)\ dt\\
=&-h^{n+3}\int_{0}^{\frac{x^*-x_0}{h}} \mu(\mu-1)\cdots(\mu-n+1)(\mu-n)(\mu-\frac{x^*-x_0}{h})\ d\mu.
\end{aligned}
}
\end{equation}
Thus, we can write that
\begin{equation}\label{res2-1}
\begin{aligned}
&-\int_{x_0}^{x^*} w(x_0,x) \frac{d}{dx}f[x_0,\cdots, x_n,x]\ dx\\
&=\frac{f^{(n+2)}(\xi_2)}{(n+2)!}h^{n+3}\int_{0}^{\frac{x^*-x_0}{h}} \mu(\mu-1)\cdots(\mu-n+1)(\mu-n)(\mu-\frac{x^*-x_0}{h})\ d\mu.
\end{aligned}
\end{equation}
Joining the partial results in (\ref{res2-1}) and (\ref{res0}), we finish the proof,
\begin{equation*}
\begin{aligned}
E_n&=\int_{x_0}^{x^*}(f(x)-p_n(x))\ dx=\int_{x_0}^{x^*}(x-x_0)\cdots(x-x_n)f[x_0,\cdots, x_n,x]\ dx\\
&=\frac{f^{(n+1)}(\xi_1)}{(n+1)!}h^{n+2}\int_0^{\frac{x^*-x_0}{h}}\mu(\mu-1)\cdots(\mu-n+1)(\mu-n)\ d\mu\\
&+\frac{f^{(n+2)}(\xi_2)}{(n+2)!}h^{n+3}\int_{0}^{\frac{x^*-x_0}{h}} \mu(\mu-1)\cdots(\mu-n+1)(\mu-n)(\mu-\frac{x^*-x_0}{h})\ d\mu,
\end{aligned}
\end{equation*}
for some $\xi_1,\xi_2\in[x_0,x_n]$.
\end{proof}
\noindent
From Lemma \ref{lema4} we can get the following corollary.

\begin{corollary}\label{cor1}
If the smooth interval is $[x^*,x_n]$:
\begin{itemize}
\item If there is not a change of sign in the polynomial of the integrand in the interval $[x^*,x_n]$, the error can be written as,
\begin{equation}\label{lemma4_1}
\begin{aligned}
E_n&=\int_{x^*}^{x_n}(f(x)-p_n(x))\ dx=(-1)^{n+2}\frac{f^{(n+1)}(\xi)}{(n+1)!}h^{n+2}\int_0^{\frac{x_n-x^*}{h}}\mu(\mu-1)\cdots(\mu-n+1)(\mu-n)\ d\mu
\end{aligned}
\end{equation}
for some $\xi\in[x_0,x_n]$.

\item If there is a change of sign in the polynomial of the integrand in the interval $[x^*,x_n]$, the error can be written as,
\begin{equation}\label{lemma4_2}
\begin{aligned}
E_n&=\int_{x^*}^{x_n}(f(x)-p_n(x))\ dx=-\frac{f^{(n+1)}(\xi_1)}{(n+1)!}h^{n+2}\int_{\frac{x^*-x_0}{h}}^{\frac{x_n-x_0}{h}}\mu(\mu-1)\cdots(\mu-n+1)(\mu-n)\ d\mu\\
&+\frac{f^{(n+2)}(\xi_2)}{(n+2)!}h^{n+3}\int_{\frac{x^*-x_0}{h}}^{\frac{x_n-x_0}{h}} \mu(\mu-1)\cdots(\mu-n+1)(\mu-n)(\mu-\frac{x_n-x_0}{h})\ d\mu.
\end{aligned}
\end{equation}
for some $\xi_1,\xi_2\in[x_0,x_n]$.
\end{itemize}
\end{corollary}

\begin{proof}
First, if there is not a change of sign in the polynomial of the integrand in the interval $[x^*,x_n]$, we just need to do the change of variables $y=x_n-x$ and proceed as in Lemma \ref{lema4},

\begin{equation*}
\begin{aligned}
E_n&=\int_{x^*}^{x_n}(f(x)-p_n(x))\ dx=\int_{x^*}^{x_n}(x-x_0)\cdots(x-x_n)f[x_0,\cdots, x_n,x]\ dx\\
&=(-1)^{n+2}\int_0^{x_n-x^*}(y-(x_n-x_0))\cdots(y-(x_n-x_{n-1}))yf[x_0,\cdots, x_n,x_n-y]\ dy\\
&=(-1)^{n+2}\frac{f^{(n+1)}(\xi)}{(n+1)!}\int_{0}^{x_n-x^*} (y-nh)\cdots(y-h)y\ dy.
\end{aligned}
\end{equation*}
for some $\xi\in[x_0,x_n]$. Applying the change of variables $y=\mu h, 0\le \mu\le n$, we can write,
\begin{equation*}
\begin{aligned}
E_n&=(-1)^{n+2}\frac{f^{(n+1)}(\xi)}{(n+1)!}h^{n+2}\int_0^{\frac{x_n-x^*}{h}}\mu(\mu-1)\cdots(\mu-n+1)(\mu-n)\ d\mu.
\end{aligned}
\end{equation*}
Secondly, if there is a change of sign in the polynomial of the integrand in the smooth interval $[x^*,x_n]$, we can define
\begin{equation}\label{weq}
\begin{aligned}
w(x_n,x)&=\int_{x_n}^{x}(t-x_0)\cdots(t-x_n)\ dt.
\end{aligned}
\end{equation} 
that satisfies, by the symmetry of the polynomials used, that at smooth zones, $$w(x_n,x_n)=w(x_n,x_0)=0, \quad w(x_n,x)<0 \quad \textrm{ for } x_0<x<x_n,$$
when $n$ is even, and 
$$w(x_n,x_n)=0, \quad w(x_n,x)<0 \quad \textrm{ for } x_0<x<x_n,$$
when $n$ is odd.

Following similar arguments to those in \cite{IK} (page 309), or just using symmetry arguments, the proof of these facts can be easily obtained.

Now, we can write the error as in Lemma \ref{lema4},
\begin{equation*}
\begin{aligned}
E_n&=\int_{x^*}^{x_n} w'(x_n,x)f[x_0,\cdots, x_n,x]\ dx.
\end{aligned}
\end{equation*}
and integrate by parts,
\begin{equation}\label{int12}
\scalemath{0.9}{
\begin{aligned}
\int_{x^*}^{x_n} w'(x_n,x) f[x_0,\cdots, x_n,x]\ dx
&=\left[w(x_n,x) f[x_0,\cdots, x_n,x]\right]_{x^*}^{x_n}-\int_{x^*}^{x_n} w(x_n,x) \frac{d}{dx}f[x_0,\cdots, x_n,x]\ dx\\
&=-w(x_n,x^*) f[x_0,\cdots, x_n,x^*]-\int_{x^*}^{x_n} w(x_n,x) \frac{d}{dx}f[x_0,\cdots, x_n,x]\ dx.
\end{aligned}
}
\end{equation}
Proceeding exactly as in Lemma \ref{lema4} and observing that,
$$w(x_n,x^*)=-w(x^*,x_n),$$
 we obtain
\begin{equation}\label{res01}
\begin{aligned}
w(x_n,x^*) f[x_0,\cdots, x_n,x^*]=-\frac{f^{(n+1)}(\xi_1)}{(n+1)!}h^{n+2}\int^{\frac{x_n-x_0}{h}}_{\frac{x^*-x_0}{h}}\mu(\mu-1)\cdots(\mu-n+1)(\mu-n)\ d\mu,
\end{aligned}
\end{equation}
for some $\xi_1\in[x_0,x_n]$.

For the last integral in (\ref{int12}) we can proceed again as in Lemma \ref{lema4} to write
\begin{equation*}
\begin{aligned}
-\int_{x^*}^{x_n} w(x_n,x) \frac{d}{dx}f[x_0,\cdots, x_n,x]\ dx&=-\int_{x^*}^{x_n}w(x_n,x) f[x_0,\cdots, x_n,x,x]\ dx\\
&=-f[x_0,\cdots, x_n,\eta_2,\eta_2]\int_{x^*}^{x_n} w(x_n,x) \ dx\\
&=-\frac{f^{(n+2)}(\xi_2)}{(n+2)!}\int_{x^*}^{x_n} \int_{x_n}^{x} (t-x_0)\cdots(t-x_n)\ dtdx,
\end{aligned}
\end{equation*}
for some $\eta_2, \xi_2\in[x_0,x_n]$.
Now we can change the order of integration and apply the change of variables $t=x_0+\mu h, 0\le \mu\le n$:
\begin{equation}\label{res3}
\begin{aligned}
-\int_{x^*}^{x_n} \int_{x_n}^{x} (t-x_0)\cdots(t-x_n)\ dtdx&=-\int_{x^*}^{x_n} \int_{t}^{x_n} (t-x_0)\cdots(t-x_n)\ dxdt\\
=&-\int_{x^*}^{x_n} (t-x_0)\cdots(t-x_n)(x_n-t)\ dt\\
=&h^{n+3}\int_{\frac{x^*-x_0}{h}}^{\frac{x_n-x_0}{h}} \mu(\mu-1)\cdots(\mu-n+1)(\mu-n)(\mu-\frac{x_n-x_0}{h})\ d\mu.
\end{aligned}
\end{equation}
Thus, we can write that
\begin{equation}\label{res2}
\begin{aligned}
&-\int_{x^*}^{x_n} w(x_n,x) \frac{d}{dx}f[x_0,\cdots, x_n,x]\ dx\\
&=\frac{f^{(n+2)}(\xi_2)}{(n+2)!}h^{n+3}\int_{\frac{x^*-x_0}{h}}^{\frac{x_n-x_0}{h}} \mu(\mu-1)\cdots(\mu-n+1)(\mu-n)(\mu-\frac{x_n-x_0}{h})\ d\mu.
\end{aligned}
\end{equation}
Joining the partial results in (\ref{res01}) and (\ref{res2}), we finish the proof,
\begin{equation*}
\begin{aligned}
E_n&=\int_{x^*}^{x_n}(f(x)-p_n(x))\ dx=\int_{x^*}^{x_n}(x-x_0)\cdots(x-x_n)f[x_0,\cdots, x_n,x]\ dx\\
&=-\frac{f^{(n+1)}(\xi_1)}{(n+1)!}h^{n+2}\int_{\frac{x^*-x_0}{h}}^{\frac{x_n-x_0}{h}}\mu(\mu-1)\cdots(\mu-n+1)(\mu-n)\ d\mu\\
&+\frac{f^{(n+2)}(\xi_2)}{(n+2)!}h^{n+3}\int_{\frac{x^*-x_0}{h}}^{\frac{x_n-x_0}{h}} \mu(\mu-1)\cdots(\mu-n+1)(\mu-n)(\mu-\frac{x_n-x_0}{h})\ d\mu,
\end{aligned}
\end{equation*}
for some $\xi_1,\xi_2\in[x^*,x_n]$.

\end{proof}

\begin{theorem}\label{teo3}
We suppose that the piecewise continuous function $f$ has singularities at $x^*$ up to the n-th derivative. The subtraction of the correction term,
\begin{equation}\label{errorteo3.1}
\scalemath{0.9}{
\begin{aligned}
C&=\int_a^{x^*}Q^+(x)\ dx+\int_{x^*}^bQ^-(x)\ dx
\end{aligned}
}
\end{equation}
to the numerical integration formula, assures that the error is:
\begin{itemize}

\item If the discontinuity is placed in the interval $[x_0,x_1]$
\begin{equation*}
\scalemath{0.9}{
\begin{aligned}
E^*(f)=E(f)+C&=C_n^1\frac{(f^-)^{(n+1)}(\xi_1)}{(n+1)!}h^{n+2}+C_n^2\frac{(f^+)^{(n+1)}(\xi_2)}{(n+1)!}h^{n+2}+C_n^3\frac{(f^+)^{(n+2)}(\xi_3)}{(n+2)!}h^{n+3},
\end{aligned}
}
\end{equation*}
with $\xi_1, \xi_2,\xi_3\in[x_0,x_n]$, and
\begin{equation*}
\scalemath{0.9}{
\begin{aligned}
C_n^1&=\int_{0}^{\frac{x^*-x_0}{h}} \mu(\mu-1)\cdots(\mu-n+1)(\mu-n)\ d\mu,\\
C_n^2&=\int_{\frac{x^*-x_0}{h}}^{\frac{x_n-x_0}{h}} \mu(\mu-1)\cdots(\mu-n+1)(\mu-n)\ d\mu.\\
C_n^3&=\int_{\frac{x^*-x_0}{h}}^{\frac{x_n-x_0}{h}} \mu(\mu-1)\cdots(\mu-n+1)(\mu-n)(\mu-n)\ d\mu.
\end{aligned}
}
\end{equation*}

\item If the discontinuity is placed in the interval $[x_{n-1},x_n]$, 
\begin{equation*}
\scalemath{0.9}{
\begin{aligned}
E(f)+C&=C_n^1\frac{(f^-)^{(n+1)}(\xi_1)}{(n+1)!}h^{n+2}+C_n^2\frac{(f^-)^{(n+2)}(\xi_2)}{(n+2)!}h^{n+3}+(-1)^{n+2}C_n^3\frac{(f^+)^{(n+1)}(\xi_3)}{(n+1)!}h^{n+2},
\end{aligned}
}
\end{equation*}
with $\xi_1,\xi_2, \xi_3\in[x_0,x_n]$, and
\begin{equation*}
\scalemath{0.9}{
\begin{aligned}
C_n^1&=\int_{0}^{\frac{x^*-x_0}{h}} \mu(\mu-1)\cdots(\mu-n+1)(\mu-n)\ d\mu,\\
C_n^2&=\int_0^{\frac{x^*-x_0}{h}} \mu(\mu-1)\cdots(\mu-n+1)(\mu-n)(\mu-\frac{x^*-x_0}{h})\ d\mu.\\
C_n^3&=\int_{0}^{\frac{x_n-x^*}{h}} \mu(\mu-1)\cdots(\mu-n+1)(\mu-n)\ d\mu.
\end{aligned}
}
\end{equation*}

\item In any other case,
\begin{equation*}
\scalemath{0.85}{
\begin{aligned}
E^*(f)=E(f)+C&=C_n^1\frac{(f^-)^{(n+1)}(\xi_1)}{(n+1)!}h^{n+2}+C_n^2\frac{(f^-)^{(n+2)}(\xi_2)}{(n+2)!}h^{n+3}+C_n^3\frac{(f^+)^{(n+1)}(\xi_3)}{(n+1)!}h^{n+2}+C_n^4\frac{(f^+)^{(n+2)}(\xi_4)}{(n+2)!}h^{n+3}.
\end{aligned}
}
\end{equation*}
with $\xi_1,\xi_2, \xi_3,\xi_4\in[x_0,x_n]$, and
\begin{equation*}
\scalemath{0.9}{
\begin{aligned}
C_n^1&=\int_{0}^{\frac{x^*-x_0}{h}} \mu(\mu-1)\cdots(\mu-n+1)(\mu-n)\ d\mu,\\
C_n^2&=\int_{0}^{\frac{x^*-x_0}{h}} \mu(\mu-1)\cdots(\mu-n+1)(\mu-n)(\mu-\frac{x^*-x_0}{h})\ d\mu,\\
C_n^3&=\int_{\frac{x^*-x_0}{h}}^{\frac{x_n-x_0}{h}} \mu(\mu-1)\cdots(\mu-n+1)(\mu-n)\ d\mu.\\
C_n^4&=\int_{\frac{x^*-x_0}{h}}^{\frac{x_n-x_0}{h}} \mu(\mu-1)\cdots(\mu-n+1)(\mu-n)(\mu-n)\ d\mu.
\end{aligned}
}
\end{equation*}

\end{itemize}
\end{theorem}
\begin{proof}
The proof is straightforward using Lemmas \ref{lema3}, \ref{lema4} and Corollary \ref{cor1}.
\end{proof}

\subsection{Correction terms for commonly used Newton-Cotes formulas}\label{CF}

In Table \ref{tabla_corr} we present some expressions for the correction terms $C$ in (\ref{errorteo3.1}). In Table \ref{tabla_corr} we have used the notation $C_{n,j}, j=1\cdots n$, being $n$ the degree of the interpolating polynomial used to obtain the integration rule. Thus, for the trapezoidal rule there is only the term $C_{1,1}$. For the Simpson's $1/3$ rule there are two terms: $C_{2,1}$ if the discontinuity falls at an odd interval and $C_{2,2}$ if the discontinuity falls at an odd interval. For the Simpson's $3/8$ rule, there are three terms: $C_{3,1}$ if $\left(\lceil\frac{x^*}{h}\rceil mod\ 3\right)=1$, $C_{3,2}$ if $\left(\lceil\frac{x^*}{h}\rceil mod\ 3\right)=2$ and $C_{3,3}$ if $\left(\lceil\frac{x^*}{h}\rceil mod\ 3\right)=0$. For higher orders, the notation is similar. Just to show an example, in Figure \ref{fig_disc3} we should use $C_{3,1}$ in the case presented to the left, $C_{3,2}$ in the case presented at the middle and $C_{3,3}$ in the case to the right.

\begin{figure}[ht]
\begin{minipage}{.2\textwidth}
\begin{center}
\resizebox{8cm}{!} {
\begin{picture}(500,150)(150,-20)
\linethickness{0.3mm}
\put(120,-10){\line(100,0){325}}
\put(192,-10){\line(0,1){5}}
\put(240,-10){\line(0,1){5}}
\put(320,-10){\line(0,1){5}}
\put(140,-20){$\cdots$}
\put(160,55){$f_{j-1}^-$}
\put(235,54){$f_{j}^+$}
\put(315,61){$f_{j+1}^+$}
\put(400,71){$f_{j+2}^+$}
\put(420,-20){$\cdots$}

\put(155,-23){$x_{j-1}$}
\put(170,3){$x^*=x_{j-1}+\alpha$}
\put(235,-23){$x_{j}$}
\put(315,-23){$x_{j+1}$}
\put(395,-23){$x_{j+2}$}


\put(160,-10){\line(0,1){55}}
\put(160,45){\circle{5}}
\put(240,-10){\line(0,1){54}}
\put(240,44){\circle{5}}
\put(320,-10){\line(0,1){61}}
\put(320,51){\circle{5}}
\put(400,-10){\line(0,1){71}}
\put(400,61){\circle{5}}

\qbezier(130,50)(160,50)(192,40)
\qbezier[80](192,40)(300,20)(400,20)
\put(325,17){$f^-_{j+1}$}
\put(400,17){$f^-_{j+2}$}
\put(240,15){$f^-_{j}$}
\qbezier(192,30)(240,55)(400,60)
\qbezier[10](160,20)(180,25)(192,30)
\put(140,20){$f^+_{j-1}$}
\put(180,70){\circle{10}}
\put(176,68){$-$}
\put(220,70){\circle{10}}
\put(216,68){$+$}
\end{picture}
}
\end{center}
\end{minipage}
\begin{minipage}{.2\textwidth}
\resizebox{8cm}{!} {
\begin{picture}(500,150)(-80,-20)
\linethickness{0.3mm}
\put(50,-10){\line(1,0){325}}
\put(80,-10){\line(0,1){5}}
\put(160,-10){\line(0,1){5}}
\put(192,-10){\line(0,1){5}}
\put(240,-10){\line(0,1){5}}
\put(55,-20){$\cdots$}
\put(80,77){$f^-_{j-1}$}
\put(160,55){$f^-_{j}$}
\put(235,54){$f^+_{j+1}$}
\put(315,61){$f_{j+2}^+$}
\put(260,-20){$\cdots$}

\put(75,-23){$x_{j-1}$}
\put(155,-23){$x_{j}$}
\put(170,3){$x^*=x_{j}+\alpha$}
\put(235,-23){$x_{j+1}$}
\put(315,-23){$x_{j+2}$}


\put(80,-10){\line(0,1){77}}
\put(80,67){\circle{5}}
\put(160,-10){\line(0,1){55}}
\put(160,45){\circle{5}}
\put(240,-10){\line(0,1){54}}
\put(240,44){\circle{5}}
\put(320,-10){\line(0,1){61}}
\put(320,51){\circle{5}}

\qbezier(60,70)(160,50)(192,35)
\qbezier[40](192,35)(250,17)(320,20)
\put(245,20){$f^-_{j+1}$}
\put(325,20){$f^-_{j+2}$}
\qbezier(192,25)(250,54)(320,50)
\qbezier[40](80,5)(135,7)(192,25)
\put(140,20){$f^+_{j}$}
\put(60,5){$f^+_{j-1}$}
\put(180,70){\circle{10}}
\put(176,68){$-$}
\put(220,70){\circle{10}}
\put(216,68){$+$}
\end{picture}
}
\end{minipage}
\begin{minipage}{.2\textwidth}
\resizebox{8cm}{!} {
\begin{picture}(500,150)(-220,-20)
\linethickness{0.3mm}
\put(50,-10){\line(1,0){325}}
\put(80,-10){\line(0,1){5}}
\put(160,-10){\line(0,1){5}}
\put(272,-10){\line(0,1){5}}
\put(240,-10){\line(0,1){5}}
\put(55,-20){$\cdots$}
\put(80,80){$f^-_{j-1}$}
\put(160,85){$f^-_{j}$}
\put(235,64){$f^-_{j+1}$}
\put(315,61){$f_{j+2}^+$}
\put(260,-20){$\cdots$}

\put(75,-23){$x_{j-1}$}
\put(155,-23){$x_{j}$}
\put(235,-23){$x_{j+1}$}
\put(250,3){$x^*=x_{j+2}-\alpha$}
\put(315,-23){$x_{j+2}$}


\put(80,-10){\line(0,1){82}}
\put(80,72){\circle{5}}
\put(160,-10){\line(0,1){81}}
\put(160,71){\circle{5}}
\put(240,-10){\line(0,1){64}}
\put(240,54){\circle{5}}
\put(320,-10){\line(0,1){61}}
\put(320,51){\circle{5}}

\qbezier(60,70)(160,85)(272,45)
\qbezier[20](272,45)(295,36)(320,25)
\put(220,20){$f^+_{j+1}$}
\put(325,20){$f^-_{j+2}$}

\qbezier(272,35)(285,38)(320,50)
\qbezier[60](80,5)(180,7)(272,35)
\put(140,10){$f^+_{j}$}
\put(60,5){$f^+_{j-1}$}
\put(260,70){\circle{10}}
\put(256,68){$-$}
\put(290,70){\circle{10}}
\put(286,68){$+$}
\end{picture}
}
\end{minipage}
\caption{Three examples of functions with singularities (solid line) placed in different intervals at a position $x^*$. We have labeled the domain to the left of the singularity as $-$ and the one to the right as $+$. We have also represented with a dashed line the prolongation of the functions through Taylor expansions at both sides of the discontinuity.}\label{fig_disc3}
\end{figure}
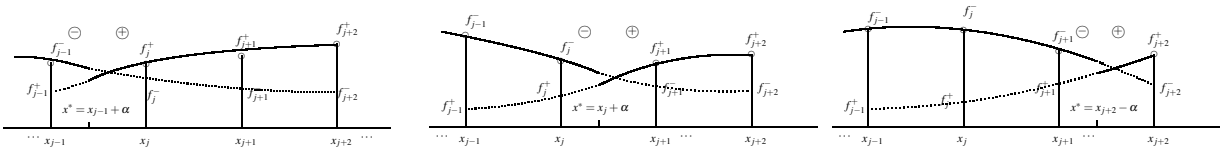

\section{Numerical experiments}\label{experimentos}

In this section we will apply the classical and corrected simple and composite trapezoid rule, Simpson's 1/3 rule and Simpson's 3/8 rule to data obtained from the discretisation of the function in (\ref{function}), that presents jumps in the function and all the derivatives. 
We will consider that we start from discretized data and that the location of the singularity, as well as the jump conditions, are known exactly. As it was motivated in the abstract and the introduction, we suppose that the function is only known at data points. 

In the grid refinement experiments, the error $E_i$ is calculated as the absolute value of the difference between the exact integral and the approximated one, obtained via the simple or composite quadrature rules. The order of accuracy is obtained in general as,
\begin{equation}\label{orden}
O_i=\frac{\ln(E_i/ E_{i+1})}{\ln(n_i/n_{i+1})},
\end{equation}
being $E_i$ the error obtained for a grid of $n_i$ points and $E_{i+1}$, the error obtained for a grid of $n_{i+1}$ points (in the experiments, $n_{i+1}=2n_i$ for the trapezoid rule and the Simpson's $\frac{3}{8}$ rule or $n_{i+1}=2n_i+1$ for the Simpson's $\frac{1}{3}$ rule).

\begin{equation}\label{function}
\scalemath{0.9}{
f(x)=\left\{\begin{array}{ll}
\cos\left(\pi x\right) + 10, & \textrm{ if } a\le x<b,\\
\sin\left(\pi x\right), & \textrm{ if } b\le x\le c.
\end{array}
\right.
}
\end{equation}
The results observed in the experiments are similar for any other piecewise continuous function that we have explored. Let us first check the numerical order attained by the simple quadrature rules. For this first experiment we initially set $a=0$, $c=0.5$. Then, we divide the interval $[a, c]$ in the number of panels used by the simple quadrature rule that we want to check. The grid-spacing is represented by $h$ and we set $b=(n+d)h$, where $d$ is a number in the interval $[0,1]$, and $n=0, 1, 2, \ldots$, depending on the number of panels that the particular rule uses. The value of $n$ and $d$ is maintained during the whole experiment. In the experiments that we present $n=0, d=0.4$, but similar results can be obtained with other values. Once we have calculated the error for the simple rule in absolute value, we divide the interval $[a, c]$ by two and we repeat the process keeping the value of $n$ and $d$. The results are presented in Figure \ref{err}. We can see that in all the cases, the error of the corrected formulas decreases following the theoretical rate and the noncorrected formulas present an error that corresponds to the first term of the corrections presented in Table \ref{tabla_corr}, that is $O(h)$. To the left of Figure \ref{err}, we present the results for the simple trapezoid rule in blue and for the corrected simple trapezoid rule in red. The error for the noncorrected rule decreases as the dashed line in blue, which shows the division of the error by two each time that the mesh size is divided by two ($O(h)$ order of accuracy). The corrected trapezoid rule behaves very similarly to the dashed line in red, which divides the error by eight when the mesh side is divided by two ($O(h^3)$ order of accuracy). At the center, the error for the non corrected Simpson's $\frac{1}{3}$ decreases as the dashed line in blue, which represents $O(h)$ order of accuracy. The error for the corrected Simpson's $\frac{1}{3}$ is represented by the dashed line in red, which represents $O(h^4)$ order of accuracy. To the right, the error for the noncorrected Simpson' $\frac{3}{8}$ rule decreases with $O(h)$ order of accuracy, while the corrected one decreases with $O(h^4)$ order of accuracy. We can also observe the numerical results in table \ref{tablaerrores_simple}.

\begin{figure}
\begin{center}
\includegraphics[height=3.5cm]{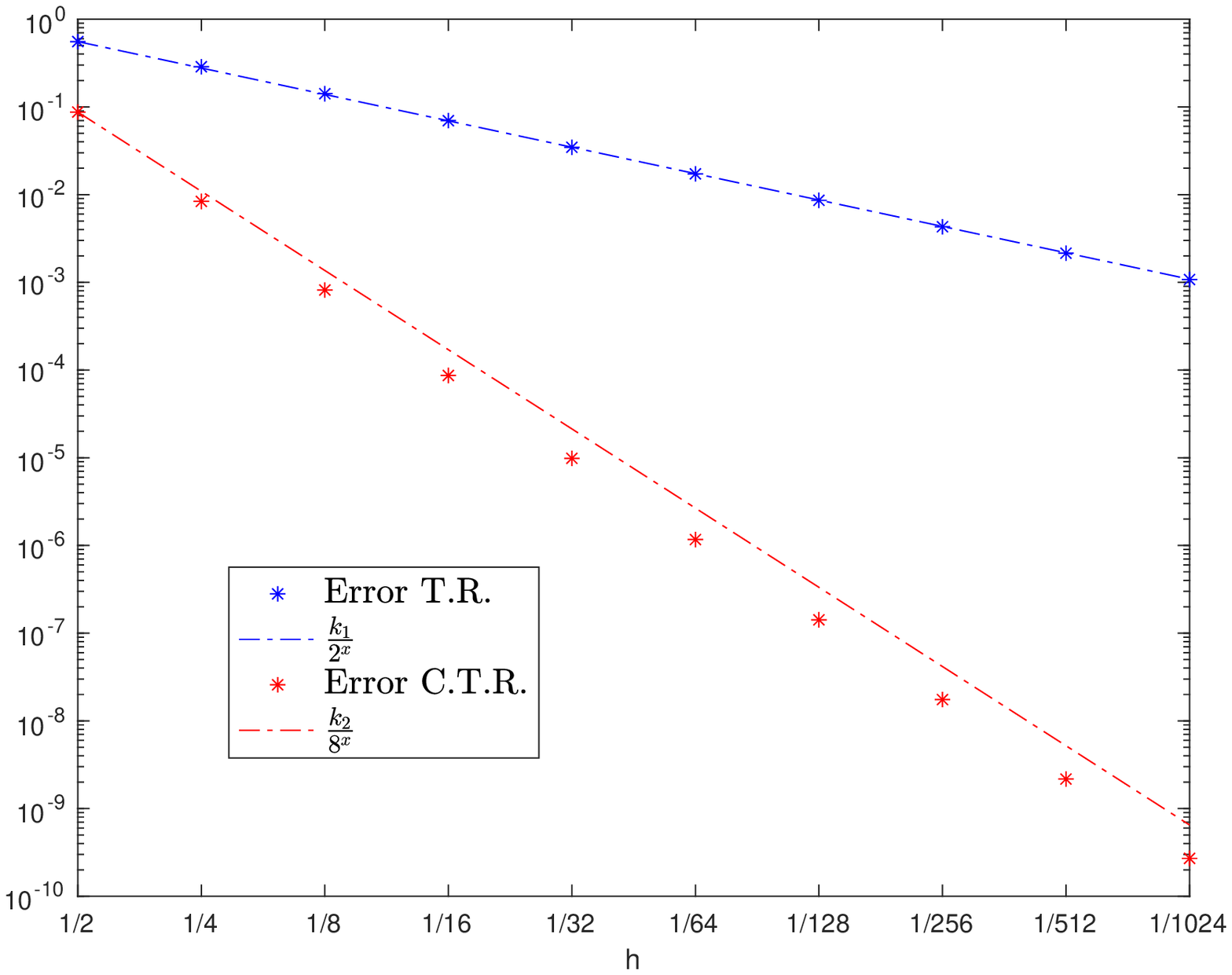}
\includegraphics[height=3.5cm]{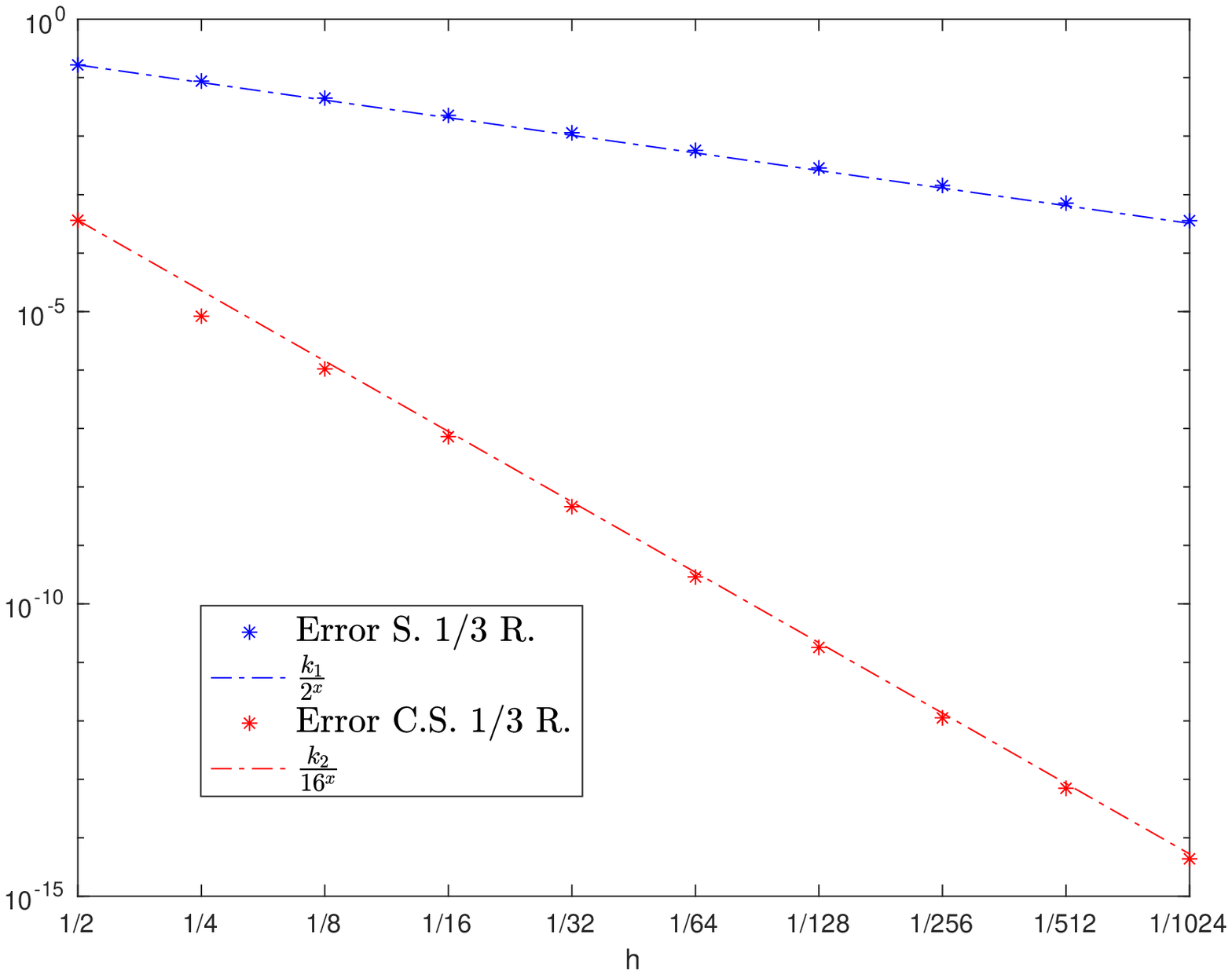}
\includegraphics[height=3.5cm]{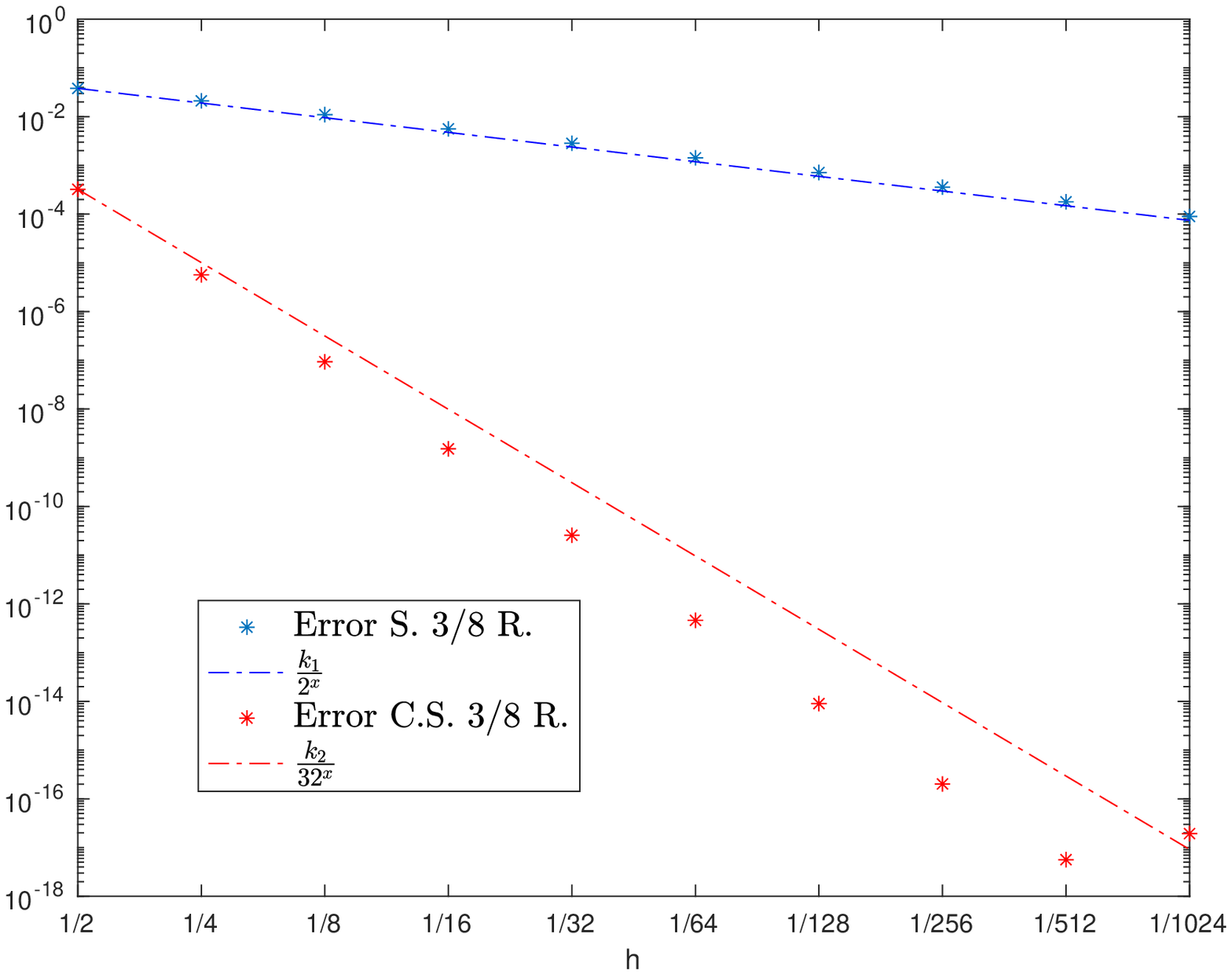}
\caption{Grid refinement analysis for the numerical integration of the function in (\ref{function}). To the left, using the simple trapezoid rule and the corrected simple trapezoid rule. At the center, using the simple Simpson's $1/3$ rule and the corrected one. To the right, using the simple Simpson's $3/8$ rule and the corrected one. In all the cases, the error of the corrected formulas decreases following the theoretical rate.}\label{err}
\end{center}
\end{figure}

\begin{figure}
\begin{center}
\includegraphics[height=3.5cm]{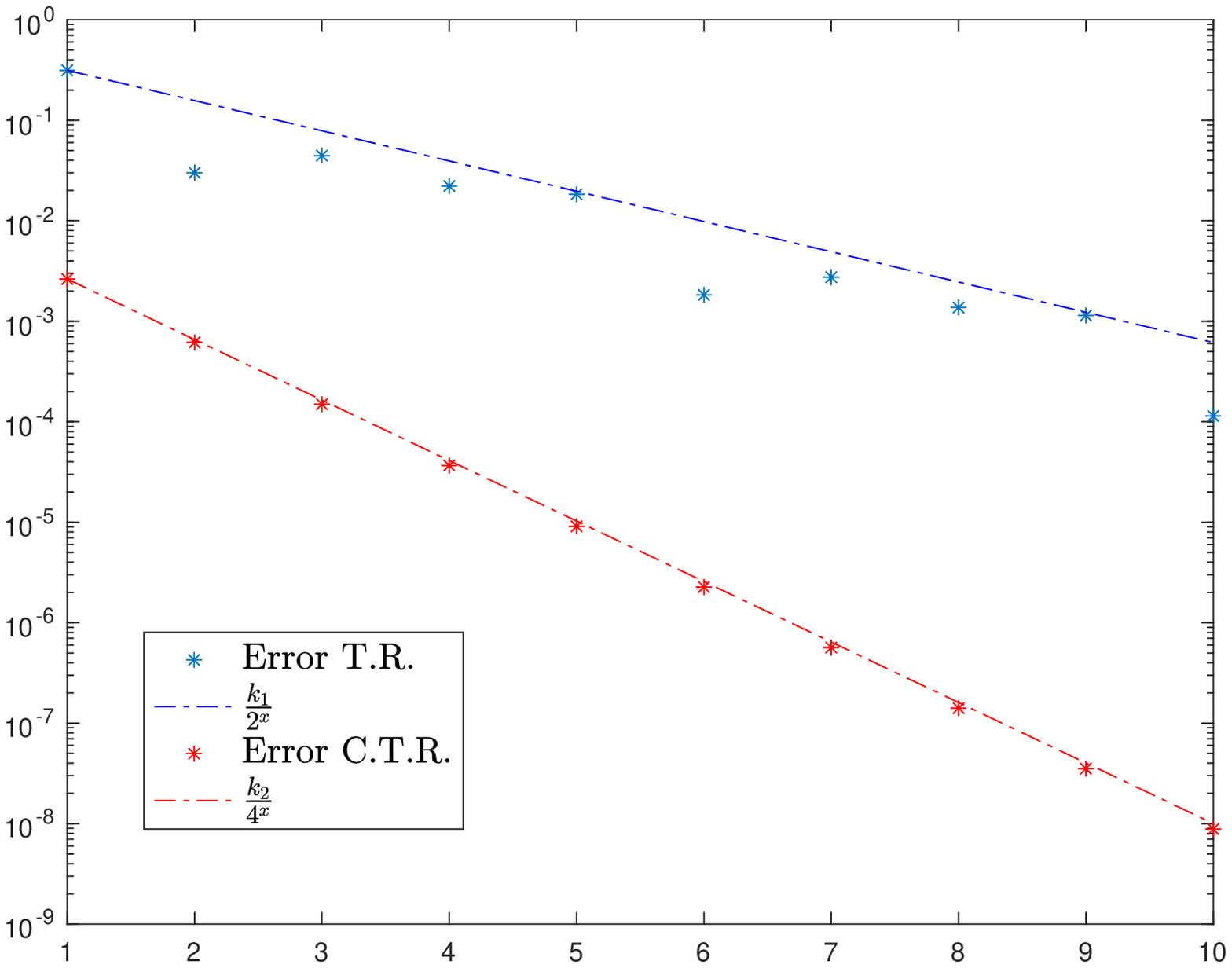}
\includegraphics[height=3.5cm]{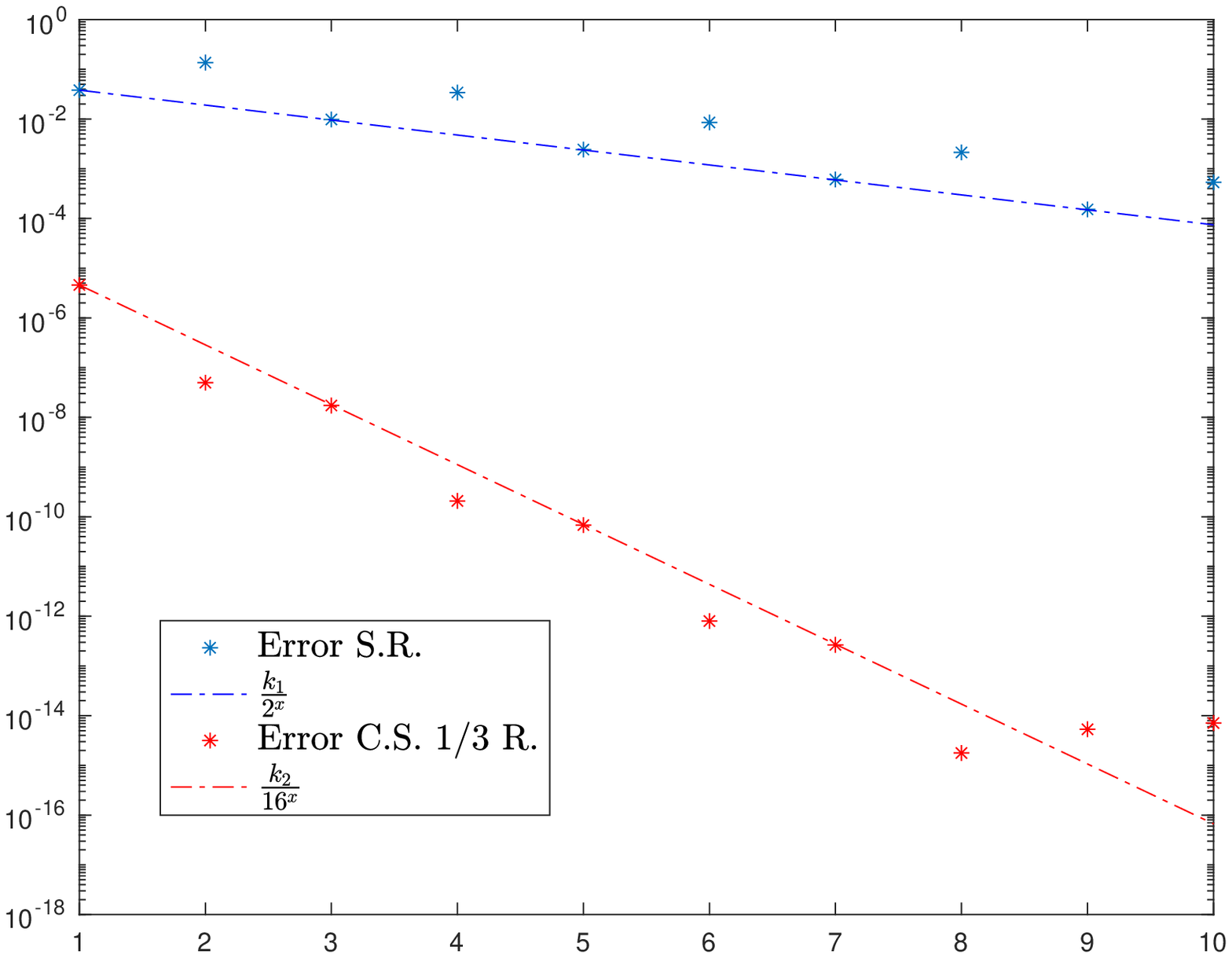}
\includegraphics[height=3.5cm]{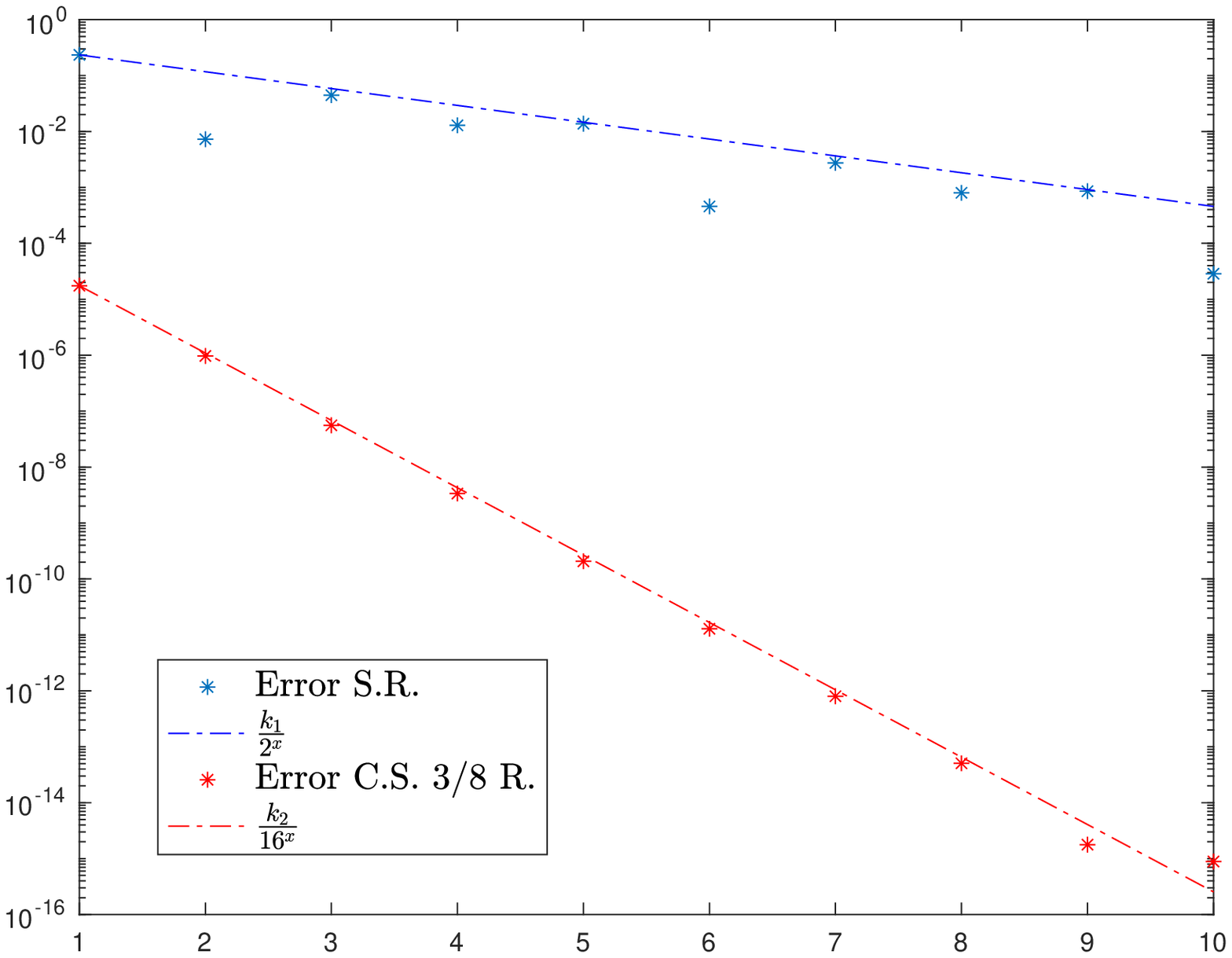}
\caption{Grid refinement analysis for the numerical integration of the function in (\ref{function}). To the left, using the composite trapezoid rule and the corrected composite trapezoid rule. At the center, using the composite Simpson's $1/3$ rule and the corrected one. To the right, using the composite Simpson's $3/8$ rule and the corrected one. In all the cases, the error of the corrected formulas decrease following the theoretical rate.}\label{plot_err}
\end{center}
\end{figure}

In Table \ref{tablaerrores} we present a second grid refinement experiment for the composite rules. 
In this case, we start from a point value discretization of the data with $n=2^i, i=3, 5, \ldots, 12$ points for the trapezoid and the Simpson's $\frac{3}{8}$ rule. For the Simpson's $\frac{1}{3}$ we set $n=2^i+1, i=3, 5, \ldots, 12$. The order presents some variability in the case of the Simpson's $\frac{1}{3}$, (as well as the order of the noncorrected rules). Even so, in Figure \ref{plot_err} we can observe that the decreasing of the errors presented in Table \ref{tablaerrores} coincides with the expected theoretical one, also represented in the graphs. In Figure \ref{plot_err} to the left, we present the results for the composite trapezoid rule in blue and for the corrected composite trapezoid rule in red. We can see that the noncorrected rule shows a decrease in the error very similar to the dashed line in blue, which shows the division of the error by two each time that the mesh size is divided by two ($O(h)$ order of accuracy). The corrected trapezoid rule behaves very similarly to the dashed line in red, which divides the error by four when the mesh side is divided by two ($O(h^2)$ order of accuracy). At the center of Figure \ref{plot_err}, the non corrected Simpson's $\frac{1}{3}$ rule behaves very similarly to the dashed line in blue, which represents $O(h)$ order of accuracy. The corrected Simpson's $\frac{1}{3}$ rule behaves very similarly to the dashed line in red, which represents $O(h^4)$ order of accuracy. Similar behavior can be observed for the Simpson' $\frac{3}{8}$ rule (at the right in Figure \ref{plot_err}): the noncorrected rule presents $O(h)$ order of accuracy, while the corrected one presents $O(h^4)$. We can see that the orders of accuracy of the corrected composite rules correspond to those of the classical composite rules at smooth zones.

\section{Conclusions}\label{conclusions}

In this article, we have presented correction terms for the classical trapezoid rule, Simpson's $\frac{1}{3}$ rule, and the most common Newton-Cotes integration formulas. These correction terms have an explicit closed formula that allows keeping the global accuracy attained by classical formulas at smooth zones even when the data contains discontinuities in the function or the derivatives. The correction terms can be used for the simple or composite classical integration formulas and it is possible to compute the integral using these formulas and then, as post-processing, add the correction terms to raise the accuracy. Correction terms for any other integration rule can be found following analogous processes to the ones shown in this work. We have also given correction terms for the most widely used Newton-Cotes quadrature formulas and we have proved that the use of these correction terms assures the expected theoretical accuracy. We have shown that the correction terms depend on the jumps of the function that is to be integrated and its derivatives. 
All the numerical experiments that we have presented, confirm the theoretical results obtained.

\begin{landscape}
\begin{table}[!ht]
\begin{center}
\resizebox{20cm}{!} {
\begin{tabular}{|c|c|c|c|}
\hline   $C_{1,1}$ & ${\frac{\left(-h+2\alpha \right)}{2}} [f]+{\frac { \left( h{\alpha}-\alpha^{2} \right) }{2}}[f']$
      \\       
\hline $C_{2,1}$ & $\left( \alpha-\frac{h}{3} \right) [f]+\frac{\alpha}{6} \left( 3
\alpha-2h \right) [f']+\frac{{\alpha}^{2}}{6} \left( \alpha-h
 \right) [f'']$
    \\
\hline $C_{2,2}$ & $-\left( \alpha-\frac{h}{3} \right) [f]+\frac{\alpha}{6} \left( 3
\alpha-2h \right) [f']-\frac{{\alpha}^{2}}{6} \left( \alpha-h
 \right) [f'']$\\
\hline $C_{3,1}$ & 
$\left( \alpha-\frac{3}{8}\,h \right) [f]+ \left( \frac{3}{8}\,h\alpha-\frac{1}{2}\,{
\alpha}^{2} \right) [f']+ \left( -\frac{3}{16}\,h{\alpha}^{2}+\frac{1}{6}\,{
\alpha}^{3} \right) [f'']+ \left( \frac{1}{16}\,h{\alpha}^{3}-\frac{1}{24}\,{
\alpha}^{4} \right) [f''']$\\
\hline $C_{3,2}$ & $\left( \alpha-\frac{1}{2}\,h \right) [f]+ \left( \frac{1}{2}\,h\alpha-\frac{1}{2}\,{
\alpha}^{2}-\frac{1}{8}\,{h}^{2} \right) [f']+ \left( -\frac{1}{4}\,h{\alpha}^{2}
-\frac{1}{48}\,{h}^{3}+\frac{1}{8}\,{h}^{2}\alpha+\frac{1}{6}\,{\alpha}^{3} \right) [f'']
+ \left( \frac{1}{48}\,{h}^{3}\alpha+\frac{1}{48}\,{h}^{4}-\frac{1}{24}\,{\alpha}^{4}-\frac{1}{16}\,{h
}^{2}{\alpha}^{2}+\frac{1}{12}\,h{\alpha}^{3} \right) [f''']$\\
\hline $C_{3,3}$ & $\left( -\alpha+\frac{3}{8}\,h \right) [f]+ \left( -\frac{1}{2}\,{\alpha}^{2}+\frac{3}{8}
\,h\alpha \right) [f']+ \left( -\frac{1}{6}\,{\alpha}^{3}+\frac{3}{16}\,h{\alpha}
^{2} \right) [f'']+ \left( -\frac{1}{24}\,{\alpha}^{4}+\frac{1}{16}\,h{\alpha}^{3
} \right) [f''']$\\
\hline $C_{4,1}$ & $\left( -{\frac {14}{45}}\,h+\alpha \right) [f]+ \left( -\frac{1}{2}\,{
\alpha}^{2}+{\frac {14}{45}}\,h\alpha \right) [f']+ \left( \frac{1}{6}\,{
\alpha}^{3}-{\frac {7}{45}}\,h{\alpha}^{2} \right) [f'']+ \left(
-\frac{1}{24}\,{\alpha}^{4}+{\frac {7}{135}}\,h{\alpha}^{3} \right) [f''']+ \left( {\frac {1}{120}}\,{\alpha}^{5}-{\frac {7}{540}}\,h{\alpha}^{
4} \right) [f^{(4)}]$\\
\hline $C_{4,2}$ & $\left( \alpha-{\frac {11}{15}}\,h \right) [f]+ \left( -{\frac {
17}{90}}\,{h}^{2}+{\frac {11}{15}}\,h\alpha-\frac{1}{2}\,{\alpha}^{2} \right)
[f']+ \left( {\frac {1}{90}}\,{h}^{3}+{\frac {17}{90}}\,{h}^{2}
\alpha-{\frac {11}{30}}\,h{\alpha}^{2}+\frac{1}{6}\,{\alpha}^{3} \right) [f'']+ \left( {\frac {11}{90}}\,h{\alpha}^{3}-\frac{1}{24}\,{\alpha}^{4}+{
\frac {11}{1080}}\,{h}^{4}-{\frac {1}{90}}\,{h}^{3}\alpha-{\frac {17}{
180}}\,{h}^{2}{\alpha}^{2} \right) [f''']+ \left( {\frac {17}{540
}}\,{h}^{2}{\alpha}^{3}+{\frac {1}{180}}\,{h}^{3}{\alpha}^{2}-{\frac {
11}{360}}\,h{\alpha}^{4}-{\frac {11}{1080}}\,{h}^{4}\alpha-{\frac {1}{
216}}\,{h}^{5}+{\frac {1}{120}}\,{\alpha}^{5} \right) [f^{(4)}]$\\
\hline $C_{4,3}$ & $\left( -\alpha+{\frac {11}{15}}\,h \right) [f]+ \left( -{\frac {
17}{90}}\,{h}^{2}+{\frac {11}{15}}\,h\alpha-\frac{1}{2}\,{\alpha}^{2} \right)
[f']+ \left( -\frac{1}{6}\,{\alpha}^{3}+{\frac {11}{30}}\,h{\alpha}^{2}-{
\frac {17}{90}}\,{h}^{2}\alpha-{\frac {1}{90}}\,{h}^{3} \right) [f'']+ \left( {\frac {11}{90}}\,h{\alpha}^{3}-\frac{1}{24}\,{\alpha}^{4}+{
\frac {11}{1080}}\,{h}^{4}-{\frac {1}{90}}\,{h}^{3}\alpha-{\frac {17}{
180}}\,{h}^{2}{\alpha}^{2} \right) [f''']+ \left( -{\frac {17}{
540}}\,{h}^{2}{\alpha}^{3}+{\frac {11}{1080}}\,{h}^{4}\alpha+{\frac {
11}{360}}\,h{\alpha}^{4}-{\frac {1}{180}}\,{h}^{3}{\alpha}^{2}+{\frac
{1}{216}}\,{h}^{5}-{\frac {1}{120}}\,{\alpha}^{5} \right) [f^{(4)}]$\\
\hline $C_{4,4}$ & 	 $\left( {\frac {14}{45}}\,h-\alpha \right) [f]+ \left( -\frac{1}{2}\,{
\alpha}^{2}+{\frac {14}{45}}\,h\alpha \right) [f']+ \left( -\frac{1}{6}\,
{\alpha}^{3}+{\frac {7}{45}}\,h{\alpha}^{2} \right) [f'']+
 \left( -\frac{1}{24}\,{\alpha}^{4}+{\frac {7}{135}}\,h{\alpha}^{3} \right) 
[f''']+ \left( -{\frac {1}{120}}\,{\alpha}^{5}+{\frac {7}{540}}\,h
{\alpha}^{4} \right) [f^{(4)}]$\\
\hline     
\end{tabular}
}
\caption{Correction terms to be subtracted from the most common integration formulas.}\label{tabla_corr}
\end{center}
\end{table}

\begin{table}[!ht]
\begin{center}
\resizebox{12cm}{!} {
\begin{tabular}{|c|c|c|c|c|c|c|c|c|c|c|c|c|}
\hline $(c-a)$ & $\frac{1}{2}$  & $\frac{1}{4}$      & $\frac{1}{8}$     & $\frac{1}{16}$    & $\frac{1}{32}$ & $\frac{1}{64}$  & $\frac{1}{128}$ &$\frac{1}{256}$ &$\frac{1}{512}$&$\frac{1}{1024}$
      \\       
\hline
Error T.R. ($E_i$)& 5.55384e-01 & 2.87374e-01 & 1.41312e-01 & 6.97373e-02 & 3.46224e-02 & 1.72492e-02 & 8.60914e-03 & 4.30072e-03 & 2.14940e-03 & 1.07446e-03 
      \\       
\hline
$O_i$ &-& 0.95056 & 1.024 & 1.0189 & 1.0102 & 1.0052 & 1.0026 & 1.0013 & 1.0006 & 1.0003
    \\
\hline     
Error C.T.R. ($E_i$)& 8.73231e-02 & 8.36720e-03 & 8.18010e-04 & 8.67118e-05 & 9.84247e-06 & 1.16748e-06 & 1.41995e-07 & 1.75028e-08 & 2.17243e-09 & 2.70590e-10
      \\       
\hline     
$O_i$& -& 3.3835 & 3.3546 & 3.2378 & 3.1391 & 3.0756 & 3.0395 & 3.0202 & 3.0102 & 3.0051
\\
\hline
\hline
Error S. 1/3 R. ($E_i$)& 1.65392e-01 & 8.75298e-02 & 4.48262e-02 & 2.26680e-02 & 1.13966e-02 & 5.71378e-03 & 2.86074e-03 & 1.43133e-03 & 7.15906e-04 & 3.58013e-04
      \\       
\hline
$O_i$&-& 0.91804 & 0.96543 & 0.98368 & 0.99206 & 0.99608 & 0.99806 & 0.99903 & 0.99952 & 0.99976
    \\
\hline     
Error C.S. 1/3 R. ($E_i$)& 3.64440e-04 & 8.31101e-06 & 1.04180e-06 & 7.25089e-08 & 4.61560e-09 & 2.88760e-10 & 1.80198e-11 & 1.12481e-12 & 7.02849e-14 & 4.37150e-15
      \\       
\hline     
$O_i$& -& 5.4545 & 2.9959 & 3.8448 & 3.9736 & 3.9986 & 4.0022 & 4.0018 & 4.0003 & 4.007 
\\
\hline
\hline
Error S. 3/8 R. ($E_i$)& 3.80713e-02 & 2.11065e-02 & 1.10144e-02 & 5.61914e-03 & 2.83719e-03 & 1.42546e-03 & 7.14440e-04 & 3.57647e-04 & 1.78930e-04 & 8.94916e-05
      \\       
\hline
$O_i$&-& 0.85102 & 0.9383 & 0.97097 & 0.98588 & 0.99304 & 0.99654 & 0.99828 & 0.99914 & 0.99957
    \\
\hline     
Error C.S. 3/8 R. ($E_i$)& 3.23113e-04 & 5.67780e-06 & 9.29742e-08 & 1.52212e-09 & 2.56990e-11 & 4.60409e-13 & 9.00668e-15 & 2.02095e-16 & 5.63785e-18 & 1.92988e-17
      \\       
\hline     
$O_i$& -& 5.8306 & 5.9324 & 5.9327 & 5.8882 & 5.8026 & 5.6758 & 5.4779 & 5.1637 & -1.7753
      \\       
\hline     
\end{tabular}
}
\caption{{Grid refinement analysis for the simple quadrature rules. The first part of the table shows the trapezoidal rule (T.R.) and the corrected trapezoidal rule (C.T.R.). The central part shows the Simpson's 1/3 Rule (S. 1/3 R.) and the corrected Simpson's rule (C.S. 1/3 R.). Finally the bottom part shows the Simpson's 3/8 Rule (S. 3/8 R.) and the corrected Simpson's 3/8 rule (C.S. 3/8 R.).  We have used the function in (\ref{function}).} 
}\label{tablaerrores_simple}
\end{center}
\end{table}
\begin{table}[!ht]
\begin{center}
\resizebox{12cm}{!} {
\begin{tabular}{|c|c|c|c|c|c|c|c|c|c|c|c|c|c|c|c|c|c|c|c|c|c|c|c|c|c|c|}
\hline $n=2^i$ & $2^4$  & $2^5$      & $2^6$     & $2^7$    & $2^8$ & $2^9$  & $2^{10}$ &$2^{11}$ &$2^{12}$&$2^{13}$
      \\       
\hline
Error T.R. ($E_i$)& 3.14564e-01 & 3.00647e-02 & 4.44922e-02 & 2.21224e-02 & 1.83580e-02 & 1.83086e-03 & 2.74401e-03 & 1.37153e-03 & 1.14264e-03 & 1.14245e-04
      \\       
\hline 
$O_i$ & -& 3.3872 & -0.56548 & 1.008 & 0.2691 & 3.3258 & -0.58377 & 1.0005 & 0.26341 & 3.3222
    \\
\hline     
Error C.T.R. ($E_i$)& 2.63164e-03 & 6.16553e-04 & 1.49283e-04 & 3.66601e-05 & 9.09938e-06 & 2.26614e-06 & 5.65427e-07 & 1.41201e-07 & 3.52844e-08 & 8.81900e-09
      \\       
\hline     
$O_i$& -& 2.0937 & 2.0462 & 2.0258 & 2.0104 & 2.0055 & 2.0028 & 2.0016 & 2.0006 & 2.0003 
\\
\hline
\hline
Error S. 3/8 R. ($E_i$)& 2.33932e-01 & 7.25343e-03 & 4.45169e-02 & 1.28813e-02 & 1.37617e-02 & 4.56764e-04 & 2.74410e-03 & 7.99969e-04 & 8.56956e-04 & 2.85576e-05
      \\       
\hline
$O_i$& -&5.0113 & -2.6176 & 1.7891 & -0.095376 & 4.9131 & -2.5868 & 1.7783 & -0.099277 & 4.9073
    \\
\hline     
Error C.S. 3/8 R. ($E_i$)& 1.74854e-05 & 9.70184e-07 & 5.58112e-08 & 3.36499e-09 & 2.07176e-10 & 1.28564e-11 & 8.00249e-13 & 5.06262e-14 & 1.77636e-15 & 8.88178e-16
      \\       
\hline     
$O_i$& -& 4.1717 & 4.1196 & 4.0519 & 4.0217 & 4.0103 & 4.0059 & 3.9825 & 4.8329 & 1
\\
\hline
\hline $n=2^i+1$ & $2^4+1$  & $2^5+1$      & $2^6+1$     & $2^7+1$    & $2^8+1$ & $2^9+1$  & $2^{10}+1$ &$2^{11}+1$ &$2^{12}+1$&$2^{13}+1$
\\
\hline
Error S. 1/3 R. ($E_i$)& 3.81374e-02 & 1.36672e-01 & 9.79905e-03 & 3.41046e-02 & 2.43374e-03 & 8.53020e-03 & 6.09443e-04 & 2.13230e-03 & 1.52298e-04 & 5.33090e-04
      \\       
\hline
$O_i$& -& -1.8414 & 3.8019 & -1.7993 & 3.8087 & -1.8094 & 3.807 & -1.8068 & 3.8074 & -1.8075
    \\
\hline     
Error C.S. 1/3 R. ($E_i$)& 4.59121e-06 & 4.99040e-08 & 1.72994e-08 & 2.08413e-10 & 6.81162e-11 & 8.00249e-13 & 2.65565e-13 & 1.77636e-15 & 5.32907e-15 & 7.10543e-15
      \\       
\hline     
$O_i$& -& 6.5236 & 1.5284 & 6.3751 & 1.6134 & 6.4114 & 1.5914 & 7.224 & -1.585 & -0.41504
      \\       
\hline     
\end{tabular}
}
\caption{{Grid refinement analysis for the composite quadrature rules. The first part of the table shows the trapezoidal rule (T.R.) and the corrected trapezoidal rule (C.T.R.). The central part shows the Simpson's 1/3 Rule (S. 1/3 R.) and the corrected Simpson's rule (C.S. 1/3 R.). Finally the bottom part shows the Simpson's 3/8 Rule (S. 3/8 R.) and the corrected Simpson's 3/8 rule (C.S. 3/8 R.). We have used the function in (\ref{function}).} 
}\label{tablaerrores}
\end{center}
\end{table}
\end{landscape}

\bibliography{bibliografia_bibdesk_no_doi_url}


\end{document}